\documentclass[11pt]{article}
\usepackage{amsmath,amssymb,amsthm, epsfig}
\usepackage{amsfonts}
\usepackage{amsmath}
\usepackage{amssymb}
\usepackage{color}
\usepackage[mathscr]{eucal}

\title{Regularity and geometric estimates for minima of discontinuous functionals
\footnote{This work is partially supported by CNPq-Brazil.}
}

\author{Raimundo Leit\~ao \quad $\&$ \quad Eduardo V. Teixeira}

\newlength{\hchng}
\newlength{\vchng}
\def \div {\mathrm{div}}
\def \dist {\mathrm{dist}}
\def \redbdry {\partial_\mathrm{red}}

\def \Leb {\mathscr{L}^n}

\newtheorem{theorem}{Theorem}[section]
\newtheorem{lemma}[theorem]{Lemma}
\newtheorem{proposition}[theorem]{Proposition}
\newtheorem{corollary}[theorem]{Corollary}

\theoremstyle{definition}
\newtheorem{definition}[theorem]{Definition}

\theoremstyle{remark}
\newtheorem{remark}[theorem]{Remark}
\numberwithin{equation}{section}

\newcommand{\intav}[1]{\mathchoice {\mathop{\vrule width 6pt height 3 pt depth  -2.5pt
\kern -8pt \intop}\nolimits_{\kern -6pt#1}} {\mathop{\vrule width
5pt height 3  pt depth -2.6pt \kern -6pt \intop}\nolimits_{#1}}
{\mathop{\vrule width 5pt height 3 pt depth -2.6pt \kern -6pt
\intop}\nolimits_{#1}} {\mathop{\vrule width 5pt height 3 pt depth
-2.6pt \kern -6pt \intop}\nolimits_{#1}}}

\begin{document}
\maketitle

\begin{abstract}

In this paper we study nonnegative minimizers of general degenerate elliptic functionals, $\int F(X,u,Du) dX \to \min$, for variational kernels $F$ that are discontinuous in $u$ with discontinuity of order $\sim \chi_{\{u > 0 \}}$. The Euler-Lagrange equation is therefore governed by a 
non-homogeneous, degenerate elliptic equation with free boundary between the positive and the zero phases of the minimizer. We show optimal gradient estimate and nondegeneracy of minima. We also address weak and strong regularity properties of free boundary. We show the set $\{u > 0 \}$ has locally finite perimeter and that the reduced free boundary, $\partial_\text{red} \{u > 0 \}$, has $\mathcal{H}^{n-1}$-total measure. For more specific problems that arise in jet flows, we show the reduced free boundary is locally the graph of a $C^{1,\gamma}$ function. 

\medskip


\medskip 


\end{abstract}


\section{Introduction}

Given a bounded smooth domain $\Omega \subset \mathbb{R}^n$ and a bounded non-negative function $\phi \in W^{1,p}(\Omega)$, $2 \leq p < n$, we study regularity and  fine geometric properties of solutions to the following minimization problem
\begin{equation}\label{Funct Intro} 
	\min \left \{\int_{\Omega} F(X,u,\nabla u) dX  : u \in W^{1,p}_\phi(\Omega)  \right \},
\end{equation}
where $W^{1,p}_\phi(\Omega)$ denotes the Sobolev space of all functions in $L^p$ with distributional derivatives in $L^p$ and trace value $\phi$. The variational kernel $F\colon \Omega \times \mathbb{R} \times \mathbb{R}^n \to \mathbb{R}$ satisfaz the following structural conditions: $F(X,u,\xi) = G(X, \xi) + g(X,u)$ and

\begin{itemize}
	\item[(G1)] For all $\xi \in \mathbb{R}^n$, the mapping $X \mapsto G(X, \xi)$ is contiunous. 
	\item[(G2)] There exists a positive constant $0 < \lambda $ such that, 
	$$
		\lambda |\xi|^p \le G(X, \xi) \le \lambda^{-1} |\xi|^p.
	$$
	\item[(G3)] For almost all $X \in \Omega$, the mapping $\xi \mapsto G(X, \xi)$ is strictly convex, differentiable and satisfaz
	$$
		G(X, t \xi) = |t|^p G(X, \xi), \quad t\in \mathbb{R}, ~ \xi \in \mathbb{R}^n.
	$$
	\item[(g1)] The function $g$ is defined by 
	$$
			g(X, u) = f(X)\left( u^{+}\right)^{m}  + Q\chi_{\left\lbrace u > 0 \right\rbrace},  \quad 1 \leq m < p,
	$$
	where $f$ is measurable,  $-K \leq f \leq K$,   for some $K>0$; $Q$ is $C^{0,\beta}$-continuous, $0 < \epsilon < Q < \epsilon^{-1}$ for some $\epsilon > 0$.
\end{itemize}
An important prototype of variational kernel to keep in mind is 
\begin{equation}\label{prototype}
	F(X, u, \xi) = |\xi|^{p-2} A(X)\xi \cdot \xi + f(X)\left( u^{+}\right)^{m} + Q \chi_{\{u > 0 \}},
\end{equation}
for a positive definite matrix $A$ with continuous coefficients. Motivations come from the study of jet flow, cavity problems, among many other applications.  For notation convenience, we label the functional appearing in the minimization problem \eqref{Funct Intro} by $\mathfrak{F} \colon W^{1,p}_\varphi(\Omega) \to \mathbb{R}$, i.e., hereafter
$$
	\mathfrak{F}(u) := \int_{\Omega} F(X,u,\nabla u) dX.
$$
Also, any positive constant $C=C\left( n, p, m, \lambda,\phi, \epsilon, K , \Omega \right)$ that depends only on dimension and the parameter constants of the problem will, hereafter, be called a \textit{universal} constant.
\par
The key feature of functional $\mathfrak{F}$ is that it is discontinuous with respect to $u$, thus the well established classical theory of the Calculus of Variations is not suitable to treat such problems. In fact, for an existing minimum $u$, the functional $\mathfrak{F}$ presents discontinuity for small perturbations near points on the, in principle unknown, set $\partial \{u > 0 \}$. Such a discontinuity reflects in a lack of smoothness of $u$ across the boundary  of its zero level surface. 

\par

The study of variational problem \eqref{Funct Intro} goes back to the fundamental work of Alt and Caffarelli, \cite{AC}, which provides a thorough analysis of such problem for $p=2$, $f\equiv 0$ and $A(X) = \text{Id}$ in \eqref{prototype}. Danielli and Petrosyan in \cite{DP} developed the corresponding Alt and Caffarelli theory for the $p$-laplace, i.e., $f\equiv 0$ and $A(X) = \text{Id}$.  

In this paper we study the variational problem \eqref{Funct Intro}  in its full generality, providing existence, regularity and geometric properties of certain heterogeneous free boundary problems ruled by degenerate elliptic equations. The results from this work are new yet for the Poisson type equation $m=1$. It also brings new results even in the linear setting $p=2$. 

\par

In Section \ref{S existence} we show there is a minimum for the functional $\mathfrak{F}$, such a minimum is non-negative and continuous in $\Omega$. We further show that within the set of positivity, $u$ satisfies the desired Euler-Lagrange equation
$$
	\text{div} \left ( \nabla_{\xi} G\left( X, Du \right) \right ) = m f(X) u^{m-1}, \quad \text{ in } \{u > 0 \},
$$
in the sense of distributions. In particular $u$ is $C^{1,\epsilon}$ in such a set. Nevertheless, due to the discontinuity of $\mathfrak{F}$ near free boundary points, $\nabla u$ jumps from positive values to zero through $\partial \{ u > 0 \}$. Therefore, the optimal regularity estimate available for minima is Lipschitz continuity. Such a result is established in Section \ref{S Lip and ND}. By Lipschitz regularity, we conclude that $u$ grows at most at a linear fashion away from the free boundary. However, from energy considerations, we  actually show that $u$ grows precisely at a linear fashion from   $\partial \{ u > 0 \}$. This is an important geometric information that provides access to finer geometric-measure features of the free boundary. In fact, in Section \ref{S Hausdorff Est} we show that 
$$
	\Lambda := \text{div} \left ( \nabla_{\xi} G\left( X, Du \right) \right ) - m f(X) u^{m-1},
$$ 
defines a non-negative measure supported along the free boundary. We further show that the set of positivity of $u$, $\{ u > 0 \}$, is locally a set of finite perimeter. A finer property is actually shown: we verify that 
$$
	\mathcal{H}^{n-1} \left (\partial \{u > 0 \} \cap B_r(Z) \right ) \sim r^{n-1},
$$
for any ball $B_r(Z) $ centered at a free boundary point. In particular we conclude the reduced free boundary, $\partial_\text{red} \{u > 0 \}$ has total $\mathcal{H}^{n-1}$-Hausdorff measure. 

In the last Section we address smoothness of the (reduced) free boundary for the heterogeneous, quasi-linear cavity problem
$$
\left \{ 
		\begin{array}{rllll}
			\div  \left ( A\left( X \right) \nabla u \right )  &=&  mf(X)u^{m-1} , & \text{ in } &  \{ u > 0 \} \\
			\langle A \nabla u, \nabla u \rangle&=& Q & \text{ on } &  \partial \{u > 0 \} \cap \Omega. 
		\end{array}
	\right.
$$
We show the free boundary is a $C^{1,\gamma}$ smooth surface, up to a possible $\mathcal{H}^{n-1}$ negligible set, providing therefore a classical solution to the corresponding quasi-linear Bernoulli type problem.

\medskip

\medskip

\noindent \textbf{Acknowledgement}  This paper is part of the first author's PhD thesis at the Department of Mathematics at Universidade Federal do Cear\'a, Brazil. Both authors would like to express their gratitude to this institution for such a pleasant and productive scientific atmosphere. This work has been partially supported by Capes and CNPq.

\section{Existence and continuity of minimizers} \label{S existence}

In this section we show the \textit{discontinuous} optimization problem \eqref{Funct Intro} has at least one minimizer. Uniqueness is known to fail even in simpler models. In the sequel we obtain a universal modulus of continuity for such a minimum. Throughout this section, we shall always work under the structural assumptions (G1)--(G3) and (g1), 
 
\begin{theorem}
\label{t 2.1}
There exists a minimizer $u \in W^{1,p}_\phi $ to the functional \eqref{Funct Intro}. Furthermore $u \ge 0$ in $\Omega$.
\end{theorem}

\begin{proof} Let us label
\begin{eqnarray*}
	I_{0}:= \min \left \{ \int_{\Omega} F(X,v,\nabla v) dX  : v \in W^{1,p}_\phi(\Omega) \right \}.
\end{eqnarray*}
Initially we show that $I_0 > - \infty$. Indeed, for any $v \in W^{1,p}_\phi(\Omega)$, by Poincar\'e inequality, Young inequality ($1 < \frac{p}{m}$) and H\"older inequality, there exist universal constants $c, C >0$ such that
\begin{eqnarray}
\label{bounded of Lp norm}
c\Vert v \Vert^{p}_{L^{p}} - c\Vert \phi \Vert^{p}_{L^{p}} - \lambda \Vert \nabla \phi \Vert^{p}_{L^{p}} \leq \lambda          \Vert \nabla v \Vert^{p}_{L^{p}},
\end{eqnarray}
and
\begin{eqnarray}
\label{th 2.1.1}
- c\Vert v \Vert^{p}_{L^{p}} - C \leq - C\Vert v \Vert^{m}_{L^{p}} \leq -K\Vert v \Vert^{m}_{L^{m}},
\end{eqnarray}
Combining (\ref{bounded of Lp norm}) and (\ref{th 2.1.1}) we obtain
\begin{eqnarray}
\label{th 2.1.b}
- C - c\Vert \phi \Vert^{p}_{L^{p}} - \lambda \Vert \nabla \phi \Vert^{p}_{L^{p}} \leq \lambda \Vert \nabla v \Vert^{p}_{L^{p}} - K\Vert v \Vert^{m}_{L^{m}},
\end{eqnarray}
which reveals
\begin{eqnarray*}
\label{}
- C - c\Vert \phi \Vert^{p}_{L^{p}} - \lambda \Vert \nabla \phi \Vert^{p}_{L^{p}} & \leq & \int_{\Omega} \left( \lambda \vert \nabla v \vert^{p} - K\vert v \vert^{m} + \epsilon  \chi_{\{v > 0 \}} \right)dX.
\end{eqnarray*}
Finally, from (G2) and (g1) we find 
\begin{eqnarray}
\label{th 2.1.a}
\int_{\Omega} \left( \lambda \vert \nabla v \vert^{p} - K\vert v \vert^{m} + \epsilon  \chi_{\{v > 0 \}} \right)dX \leq\int_{\Omega} F(X,v,\nabla v) dX.
\end{eqnarray}
\par
Let $v_{j} \in W^{1,p}_\phi(\Omega)$ be a minimizing sequence. We can suppose for $j \gg 1$, that    
\begin{eqnarray*}
\int_{\Omega} F(X,v_{j},\nabla v_{j}) dX & \leq & I_{0} + 1. 
\end{eqnarray*}
From (\ref{th 2.1.a}) and the H\"older inequality we obtain
\begin{eqnarray}
\label{th 2.1.2}
\int_{\Omega}\vert \nabla v_{j} \vert^{p} dX & \leq & \frac{K}{\lambda}\Vert v_{j} \Vert^{m}_{L^{m}} + \frac{\epsilon}{\lambda} \Leb \left( \Omega \right) + \frac{I_{0}}{\lambda} + \frac{1}{\lambda} \\ \nonumber
& \leq & C\Vert v_{j} \Vert^{m}_{L^{p}} + \frac{I_{0}}{\lambda} + C 
\end{eqnarray}
By Poincar\'e inequality we have
\begin{eqnarray}
\label{th 2.1.3}
C \Vert v_{j} \Vert^{m}_{L^{p}} & \leq & C\left(\Vert \nabla v_{j} \Vert^{m}_{L^{p}} + \Vert \nabla \phi \Vert^{m}_{L^{p}} + \Vert \phi \Vert^{m}_{L^{p}} \right).  
\end{eqnarray}
Also we have
\begin{eqnarray}
\label{th 2.1.4}
C\Vert \nabla v_{j} \Vert^{m}_{L^{p}} \leq C + \frac{1}{2} \Vert \nabla v_{j} \Vert^{p}_{L^{p}}.
\end{eqnarray}
Combining (\ref{th 2.1.2}), (\ref{th 2.1.3}) and (\ref{th 2.1.4}) we obtain
\begin{eqnarray}
\label{bounded of Lp norm of gradient}
\int_{\Omega}\vert \nabla v_{j} \vert^{p} dX & \leq & C\left(\Vert \nabla \phi \Vert^{m}_{L^{p}} + \Vert \phi \Vert^{m}_{L^{p}} \right) + \frac{I_{0}}{\lambda} + C.
\end{eqnarray}
Thus, using  Poincar\'e inequality once more, we conclude that $\{v_{j} - \phi\}$ is a bounded sequence in $W_0^{1,p}\left( \Omega \right)$. By weak compactness, there is a function $u \in W_{\phi}^{1,p}\left( \Omega \right)$ such that, up to a subsequence,
\begin{eqnarray*}
v_{j} \rightarrow  u  \ \ \mbox{weakly} \ \mbox{in} \ W^{1,p}\left( \Omega \right), \quad v_{j} \rightarrow  u \ \ \mbox{in} \ L^{p}\left( \Omega \right) \quad v_{j} \rightarrow u  \ \ \text{a.  e.} \ \mbox{in} \ \Omega. 
\end{eqnarray*}
By compactness, for a subsequence, we obtain
\begin{eqnarray*} 
v_{j} \rightarrow  u \ \ \mbox{in} \ L^{p}\left( \Omega \right).
\end{eqnarray*}
It now follows from lower semicontinuity of $G$, see, for instance, \cite{HKM}, Chap. 5), 
\begin{eqnarray*}
\int_{\Omega}G\left( X, \nabla u \right)dX &\leq & \liminf_{j \rightarrow \infty} \int_{\Omega}G\left( X, \nabla v_{j} \right)dX.
\end{eqnarray*}
Condition (g1) and pointiwise convergence gives
\begin{eqnarray*}
\int_{\Omega}g\left( X,  u \right)dX &\leq & \liminf_{j \rightarrow \infty} \int_{\Omega}g\left( X, v_{j} \right)dX.
\end{eqnarray*}
In conclusion,
\begin{eqnarray*}
\int_{\Omega}F\left( X,  u, \nabla u \right)dX &\leq & \liminf_{j \rightarrow \infty} \int_{\Omega}F\left( X, v_{j}, \nabla v_{j} \right)dX,
\end{eqnarray*}
which proves the existence a minimizer. 

\par

Let us  turn our attention to non-negativity  property of  $u$. To verify this fact, we initially notice that
$$
	\chi_{\left\lbrace \max\left( u, 0\right)> 0\right\rbrace} \leq \chi_{\left\lbrace  u > 0 \right\rbrace}.
$$
Thus, 
\begin{eqnarray}
\label{t 2.1.1}
\int_{\Omega}  g\left( X, \max \left( u, 0\right)\right) - g\left( X, u\right) dX  \nonumber 
& = & \int_{\Omega} f \left( \left( u^{+}\right)^{m} - \left( u^{+}\right)^{m} \right) dX  \\ \nonumber
& + & \int_{\Omega} Q\left( \chi_{\left\lbrace \max\left( u, 0\right)> 0\right\rbrace}  - \chi_{\left\lbrace  u > 0 \right\rbrace}\right)  dX \\ 
& \leq & 0.
\end{eqnarray}
Then, by minimality of $u$ and (\ref{t 2.1.1}) we obtain
\begin{eqnarray}
0 \nonumber &\leq & \int_{\Omega}  F(X, \max \left( u, 0\right), \nabla \left( \max \left( u, 0\right)\right) ) - F(X, u, \nabla u) dX \\ \nonumber
& = & \int_{\Omega}  G(X, \nabla \left( \max \left( u, 0\right)\right) ) - G(X, \nabla u) dX  + \int_{\Omega} g\left( X, \max \left( u, 0\right)\right) - g\left( X, u\right) dX \\ 
& \leq &  - \int_{\left\lbrace u < 0 \right\rbrace } G(X, \nabla u) dX.  
\end{eqnarray} 
From (G2) we can write,
\begin{eqnarray}
0 \nonumber & \geq & \int_{\left\lbrace u \leq 0 \right\rbrace} G\left(X, \nabla u \right) dX \\ \nonumber
& \geq &\lambda \int_{\left\lbrace u < 0 \right\rbrace} \vert \nabla u \vert^{p} dX
\\ \nonumber
& = & \lambda \int_{\Omega} \vert \nabla \left( \min\left( u, 0\right)\right)\vert^{p} dX,   
\end{eqnarray} 
and the nonnegativity of $u$ follows.  
\end{proof}

\bigskip

\begin{remark}

As previously set in condition (g1), throughout the whole paper we shall work under the range $1\le m < p$. Such a constrain is merely for the existence of minima of the functional $\mathfrak{F}$. Also, from inequalities (\ref{bounded of Lp norm}), (\ref{th 2.1.2}) and (\ref{th 2.1.3}), it is possible to show existence of minimizer provided $K$ is small enough. In addition, we can  to obtain critical points to the functional $\mathfrak{F}$,  within the range $p < m \leq p*$ (see \cite{IPeral}), where $p*:= \frac{np}{n-p}$. 
\end{remark}

Even though the functional is discontinuous, we will show that by energy considerations, it is possible to prove that minimizers are universally continuous. The delicate question of optimal regularity will be addressed in the next Section.

\begin{theorem} 
\label{th 2.2}
Let $u$ be a minimizer of (\ref{Funct Intro}). There exist universal constants $M > 0$ and $\beta \in \left( 0, 1\right)$ such that $\Vert u \Vert_{L^{\infty}\left( \Omega \right)} \leq  M$ and $u \in C_{loc}^{0,\beta}\left( \Omega \right)$.
\end{theorem}
\begin{proof}
Let us assume that we have already established boundedness of minimizer, i.e. $\Vert u \Vert_{L^{\infty}\left( \Omega \right)} \leq  M$. For
\begin{eqnarray*}
\mathcal{A}\left( X, \xi \right):= \nabla_{\xi} G\left( X, \xi \right),   
\end{eqnarray*}
let $h$ be the solution of boundary value problem
\begin{eqnarray}
\label{th 2.2.1}
\left \{
\begin{array}{rcl}
\div \left( \mathcal{A}\left( X, \nabla h \right) \right)& = & 0  \ \mbox{in} \ B\\
h & = & u \ \mbox{on} \ \partial B, \\   
\end{array}
\right.
\end{eqnarray}
where $B \Subset \Omega$ is fixed ball. By minimality of $u$, (g1) and Mean Value Theorem we have
\begin{eqnarray}
\label{th 2.2.2}
\int_{B}  G(X, \nabla u ) - G(X, \nabla h) dX  & \leq &  \int_{B}  g\left( X, h \right) - g\left( X, u \right) dX  \\ \nonumber 
& = & \int_{B}  f \left(  h^{m} - u^{m} \right) dX  +  \int_{B} Q\left( \chi_{\left\lbrace h > 0\right\rbrace}  - \chi_{\left\lbrace  u > 0 \right\rbrace}\right)  dX \\ \nonumber
& \leq & mKM^{m-1} \int_{B} \vert u - h \vert dX  + \epsilon^{-1}\Leb \left( \left\lbrace u = 0 \right\rbrace \cap B \right).
\end{eqnarray}
We have use that an $\mathcal{A}$-harmonic function with nonnegative boundary  value is  positive and $0 \leq u,h \leq M$ 
(see \cite{HKM}, Chap. 3, Prop. 3.24). From (G3) we have 
\begin{eqnarray}
\label{th 2.2.3}
\langle \mathcal{A}\left( X, \xi \right), \xi \rangle = p G\left( X, \xi \right), 
\end{eqnarray} 
for a.e. $X \in \Omega$ and all $\xi \in \mathbb{R}^{n}$. Thus, from monotonicity, see for instance,  Lemma 3.2 in \cite{Teix01}, we obtain 
\begin{eqnarray}
\label{th 2.2.4}
 \int_{B}  G(X, \nabla u ) - G(X, \nabla h) dX \geq c 
 \int_{B} \vert \nabla \left( u - h\right) \vert^{p} dX 
\end{eqnarray}
where $c=c\left(n,p,G \right)$ is a positive constant. Young inequality and Poincar\'e inequality together yield 
\begin{eqnarray}
\label{th 2.2.5}
mKM^{m-1} \int_{B} \vert u - h \vert dX \leq  \frac{c}{2}\int_{B} \vert \nabla \left( u - h \right) \vert^{p} dX + C\Leb \left( B \right). 
\end{eqnarray}
Thus, if $B$ is a ball of radius $r > 0$, it follows from (\ref{th 2.2.2}), (\ref{th 2.2.4}) and (\ref{th 2.2.5}) that
\begin{eqnarray}
\label{t 2.1.5}
\Vert \nabla \left( u - h\right) \Vert_{L^{p}\left( B \right)} \leq 
 Cr^{\frac{n}{p}}.
\end{eqnarray}
Hence, from Morrey's Theorem (recall $h$ is H\"older continuous by elliptic estimates)  there is a constant $\beta=\beta\left( n, p \right) > 0$ such that $u \in C^{0,\beta}_{\text{loc}}(\Omega)$. 

\par

Let us now turn our attention to $L^\infty$ bounds of  $u$. Let is label 
$$
	j_{_{0}}  := \left \lceil \sup_{\partial \Omega}\phi  \right \rceil, 
$$	
that is, the smallest natural number above $\sup_{\partial \Omega}\phi$.	For each $j \geq j_{0}$ we define the truncated function $u_{j}: \Omega \rightarrow \mathbb{R}$ by
\begin{eqnarray}
\label{th 2.2.6}
u_{j}= \left \{
\begin{array}{rcl}
j, \ \ \mbox{if} \ \ u > j, \\
u, \ \ \mbox{if} \ \ u \leq j.\\
\end{array}
\right. 
\end{eqnarray} 
Clearly, by the choice of $j_{0}$,  $u_{j}\in W_{\phi}^{1,p}\left( \Omega \right)$ and 
$$
	\left\lbrace u_{j} > 0 \right\rbrace = \left\lbrace u > 0 \right\rbrace. 
$$
If we denote $A_{j} := \left\lbrace u > j \right\rbrace$, we have, for each $j > j_0$
\begin{eqnarray}
u = u_{j} \ \ \mbox{in} \ \ A^{c}_{j}\ \ \mbox{and} \ \ u_{j} = j  \ \ \mbox{in} \ \ A_{j}.
\end{eqnarray}
Thus, by minimality of $u$ and (G2), there holds
\begin{eqnarray}
\label{th 2.2.7}
\lambda \int_{A_{j}} \vert \nabla u \vert^{p} dX & \leq & \int_{A_{j}}  G(X, \nabla u ) \\ \nonumber & = &  \int_{\Omega}  G(X, \nabla u ) - G(X, \nabla u_{j}) dX \\ \nonumber
& \leq &  \int_{\Omega}  f \left(  u_{j}^{m} - u^{m} \right) dX  \\ \nonumber
& = & \int_{A_{j}}  f \left(  u^{m} - j^{m} \right)dX.
\end{eqnarray}
Taking into account the elementary inequality
\begin{eqnarray}
u^{m} = \left( u - j + j \right)^{m} \leq 2^{m}\left[ \left( u - j \right)^{m} + j^{m} \right], 
\end{eqnarray}  
we obtain
\begin{eqnarray}
\label{th 2.2.8}
\int_{A_{j}} \vert f \vert \left(  u^{m} - j^{m} \right)dX & \leq & \nonumber C\int_{A_{j}} \vert f \vert \left( u - j \right)^{m}dX + Cj^{m}\Leb \left( A_{j} \right) \\  
& = & C\int_{A_{j}} \vert f \vert \left[ \left( u - j \right)^{+}\right] ^{m}dX + Cj^{m}\Leb \left( A_{j} \right).
\end{eqnarray} 
From the range of truncation we consider, it follows  that $\left( u - j \right)^{+} \in W_{0}^{1,p}\left( \Omega \right)$. Hence, applying H\"older inequality and Gagliardo-Nirenberg inequality (see \cite{GT}, Chap. 7), we find
\begin{eqnarray}
\label{th 2.2.9}
\int_{A_{j}} \vert f \vert \left[ \left( u - j \right)^{+}\right]^{m}dX & \leq & \nonumber K \left[ \Leb \left( A_{j} \right)\right]^{1- \frac{m}{p*}} \Vert \nabla u \Vert_{L^{p}\left( A_{j}\right)}^{m}. \nonumber
\end{eqnarray}
Young inequality gives, then, 
\begin{eqnarray}
\label{th 2.2.10}
K \left[ \Leb \left( A_{j} \right)\right]^{1- \frac{m}{p*}} \Vert \nabla u \Vert_{L^{p}\left( A_{j}\right)}^{m} \leq C\left[ \Leb \left( A_{j} \right)\right]^{\frac{p}{p-m}- \frac{pm}{p*\left( p-m\right)}} + \frac{\lambda}{2}\Vert \nabla u \Vert_{L^{p}\left( A_{j} \right)}^{p} .
\end{eqnarray} 
Combining (\ref{th 2.2.7}), (\ref{th 2.2.8}) and (\ref{th 2.2.10}) we obtain
\begin{eqnarray}
\int_{A_{j}} \vert \nabla u \vert^{p} dX \leq C j^{m} \left[ \Leb \left( A_{j} \right)\right]^{1 - \frac{p}{N} + \varepsilon},
\end{eqnarray} 
where $\varepsilon = \frac{p^{2}}{n\left( p - m \right) }$ and (see (\ref{bounded of Lp norm}) and (\ref{bounded of Lp norm of gradient}) substituting $I_{0}$ by $\int_{\Omega} F(X, \phi, \nabla \phi) dX$)
\begin{eqnarray}
\Vert u \Vert_{L^{1}\left( A_{j_{0}} \right)} \leq \left[ \Leb \left( A_{j_{0}} \right)\right]^{\frac{p-1}{p}}\Vert u \Vert_{L^{p}\left( A_{j_{0}} \right)} \leq C.
\end{eqnarray}
Boundedness of $u$ now follows from a general machinery, see for instance,  \cite{ON}, Chap. 2, Lemma 5.2, Page 71.
\end{proof}
\bigskip

At this stage of the program, an important consequence of Theorem \ref{th 2.2} is the fact that the positivity set of $u$, $\{ u > 0 \}$ is open. Next Theorem gives the Euler Lagrange equation satisfies therein. 

 \begin{theorem} 
\label{t 2.3}
Let $u$ be a minimizer of (\ref{Funct Intro}). Within the open set $\{u > 0 \}$, $u$ satisfies
$$
	\div  \left ( \mathcal{A}\left( X, \nabla u \right) \right ) = mf(X)u^{m-1}
$$
in the distributional sense.
\end{theorem}

\begin{proof}
Fixed  $\zeta \in C_{0}^{\infty}\left( \lbrace u > 0 \rbrace \right)$, there is a $0 < \varepsilon_0 \ll 1$, sufficiently small such that
\begin{eqnarray}
\left\lbrace u + \varepsilon \zeta > 0 \right\rbrace = \left\lbrace u > 0 \right\rbrace, 
\end{eqnarray}
for all $ 0 < \varepsilon \le \varepsilon_0.$ We can easily write,
\begin{eqnarray*}
\frac{1}{\varepsilon}\int\limits_{\left\lbrace u > 0 \right\rbrace}  F(X, u + \varepsilon \zeta , \nabla \left( u + \varepsilon \zeta \right) ) - F(X, u, \nabla u)  & = & \nonumber \frac{1}{\varepsilon}\int\limits_{\left\lbrace u > 0 \right\rbrace}  G(X, \nabla \left( u + \varepsilon \zeta \right) ) - G(X, \nabla u)   \\ 
& + & \int\limits_{\left\lbrace u > 0 \right\rbrace } f\left( X \right) \frac {\left( u + \varepsilon \zeta \right)^{m} - u^{m} }{\varepsilon}. 
\end{eqnarray*}
Taking $\varepsilon \rightarrow 0$, and using the minimality of $u$, we obtain
\begin{eqnarray*}
0 & = & \nonumber \int_{\left\lbrace u > 0 \right\rbrace} \nabla_{\xi} G\left( X, \nabla u \right)\cdot \nabla \zeta dX + \int_{\left\lbrace u > 0 \right\rbrace}  mf\left( X \right)u^{m-1}\zeta dX \\ \nonumber
\end{eqnarray*}  
and the result follows.
\end{proof}
\vspace{0,5cm}
\section{Upper and lower gradient bounds} \label{S Lip and ND}

In the previous Section we have shown minimizers are $C^{0,\beta}$ continuous in $\Omega$, for some unknown $\beta < 1$. From the discontinuity of the functional $\mathfrak{F}$, along the free-surface, it is also possible to check that minimizers are not $C^1$-regular through the zero level surface $\partial \{u > 0 \}$. Thus the optimal regularity one should hope for $u$ is Lipschitz continuity. This is the contents of next Theorem.

\begin{theorem} 
\label{t 3.1}
Given a subdomain $\Omega' \Subset \Omega$, there exists a constant $C>0$ that depends only on $\Omega'$ and universal constants, such that
$$
	\| \nabla u \|_{L^\infty(\Omega')} \le C.
$$ 
\end{theorem}

\begin{proof}
Let us suppose, for the purpose of contradiction, that there exists a sequence of points $X_{j} \in \Omega' \cap \left\lbrace u > 0 \right\rbrace$ such that
\begin{eqnarray}
\label{t 3.1.1}
X_{j} \rightarrow \partial \{ u > 0 \} \ \ \ \mbox{and} \ \ \ \frac{u\left( X_{j} \right)}{\dist\left( X_{j}, \partial \{ u > 0 \} \right)}\nearrow \infty.  
\end{eqnarray}
We denote 
\begin{eqnarray*}
U_{j}:= u\left( X_{j} \right) \ \ \ \mbox{and} \ \ \ d_{j}:= \dist\left( X_{j}, \partial \{ u > 0 \} \right).  
\end{eqnarray*}
For each $j$, let $Y_{j} \in \partial\{ u > 0 \}$  be such that 
\begin{eqnarray*}
d_{j} = \vert X_{j} - Y_{j}\vert.  
\end{eqnarray*}
Recall we have proven in Theorem \ref{t 2.3} that
\begin{eqnarray*}
\div \left( \mathcal{A}\left( X, \nabla u \right) \right) = mf \left( X \right)u^{m-1} \ \ \mbox{in} \ \ \left\lbrace u > 0 \right\rbrace.
\end{eqnarray*}
Thus, by Harnack's inequality, universal boundedness of $u$ and (g1), there exist  universal constants  $C, ~ c>0$, such that
\begin{eqnarray*}
C d_{j} + \inf\limits_{B_{\frac{3}{4}d_{j}}\left( X_{j}\right)} u \geq cU_{j}.  
\end{eqnarray*}
In turn, we have
\begin{eqnarray}
\label{t 3.1.a}
\sup\limits_{B_{\frac{1}{4}d_{j}}\left( Y_{j}\right)} u \geq cU_{j} - C d_{j}.
\end{eqnarray}
Consider the set
\begin{eqnarray}
\label{t 3.1.c}
A_{j}:= \left\lbrace Z \in B_{d_{j}}\left( Y_{j}\right): \dist \left( Z, \partial \left\lbrace u > 0 \right\rbrace \right) \leq \frac{1}{3}\dist \left( Z, \partial B_{d_{j}}\left( Y_{j}\right) \right) \right\rbrace. 
\end{eqnarray}
First we claim that $B_{\frac{d_{j}}{4}}\left( Y_{j}\right) \subset A_{j}$. In fact, if $\vert Z - Y_{j}\vert \leq \frac{d_{j}}{4}$, then
\begin{eqnarray}
\label{}
\frac{1}{3}\dist \left( Z, \partial B_{d_{j}}\left( Y_{j}\right) \right) \geq \frac{1}{3}\frac{3d_{j}}{4} = \frac{d_{j}}{4} \geq \dist \left( Z, \partial \left\lbrace u > 0 \right\rbrace \right). 
\end{eqnarray}
Thus,  
\begin{eqnarray} 
\label{}
M_{j} &:=& \sup_{Z \in A_{j}} \dist \left( Z, \partial B_{d_{j}}\left( Y_{j}\right) \right)u\left( Z\right) \\
	&=&  \dist \left( Z_{j}, \partial B_{d_{j}}\left( Y_{j}\right) \right)u\left( Z_{j}\right) \\
	&\geq& \frac{3}{4}\sup\limits_{B_{\frac{1}{4}d_{j}}\left( Y_{j}\right)} u. 
\end{eqnarray}
Therefore,
\begin{eqnarray}
\label{}
u\left( Z_{j}\right) \geq \frac{d_{j}}{\dist \left( Z_{j}, \partial B_{d_{j}}\left( Y_{j}\right) \right)}\frac{3}{4}\sup\limits_{B_{\frac{1}{4}d_{j}}\left( Y_{j}\right)} u \geq \frac{3}{4}\sup\limits_{B_{\frac{1}{4}d_{j}}\left( Y_{j}\right)} u. 
\end{eqnarray}
Hence, using (\ref{t 3.1.a}) we have
\begin{eqnarray}
\label{t 3.1.i}
u\left( Z_{j}\right)\geq \frac{3}{4}\left( cU_{j} - C d_{j}\right). 
\end{eqnarray}
For each $j$, let $W_{j} \in \partial \left\lbrace u > 0 \right\rbrace$ be such that
\begin{eqnarray}
\label{t 3.1.f}
r_{j}:= \vert Z_{j} - W_{j} \vert = \dist \left( Z_{j}, \partial \left\lbrace u > 0 \right\rbrace \right) \leq \frac{1}{3}\dist \left( Z_{j}, \partial B_{d_{j}}\left( Y_{j}\right) \right). 
\end{eqnarray}
Using (\ref{t 3.1.f}) we conclude that
\begin{eqnarray}
\label{t 3.1.g}
r_{j}\leq \frac{1}{3}\left(d _{j} - \vert Z_{j} - Y_{j} \vert \right) \leq \frac{1}{3}\left( d_{j} - r_{j}\right). 
\end{eqnarray}
That is,
\begin{eqnarray}
\label{t 3.1.h}
\frac{1}{r_{j}} \geq \frac{4}{d_{j}}. 
\end{eqnarray}
From (\ref{t 3.1.i}) and (\ref{t 3.1.h}) we have, for $j$ sufficiently large, as to
$$
	\frac{U\left( j\right)}{d_{j}} \geq \frac{C}{c},
$$
the following lower estimate
\begin{eqnarray}
\label{j}
\frac{u\left( Z_{j}\right)}{r_{j}}\geq \frac{1}{r_{j}}\frac{3}{4}\left( cU_{j} - C d_{j}\right)\geq 3 \left( c\frac{U_{j}}{d_{j}} - C \right). 
\end{eqnarray}
We have proven that  
\begin{eqnarray}
\label{t 3.1.j}
\frac{u\left( Z_{j}\right)}{r_{j}} \rightarrow \infty. 
\end{eqnarray}
If $X \in B_{2r_{j}}\left( W_{j}\right)$ we obtain, see (\ref{t 3.1.f}), 
\begin{eqnarray}
\label{}
\vert X - Y_{j} \vert \leq \vert X - W_{j} \vert + \vert W_{j} - Z_{j} \vert + \vert Z_{j} - Y_{j} \vert \leq 2r_{j} + r_{j}+ \vert Z_{j} - Y_{j} \vert \leq d_{j}.  
\end{eqnarray}
Thus, $B_{2r_{j}}\left( W_{j}\right) \subset B_{d_{j}}\left( Y_{j} \right)$. Also we have
\begin{eqnarray}
\label{}
\dist\left( X, \partial \left\lbrace u > 0\right\rbrace \right) \leq \frac{r_{j}}{2},
\end{eqnarray} 
for all $X \in B_{\frac{r_{j}}{2}}\left( W_{j}\right)$. Triangular inequality and (\ref{t 3.1.f}) then yield
\begin{eqnarray}
\label{}
\dist\left( X, \partial B_{d_{j}}\left( Y_{j}\right) \right) &\geq & \nonumber
\dist\left( Z_{j}, \partial B_{d_{j}}\left( Y_{j}\right) \right) - \vert Z_{j} - X \vert \\ \nonumber
& \geq & \dist\left( Z_{j}, \partial B_{d_{j}}\left( Y_{j}\right) \right) - \frac{3r_{j}}{2}\\ \nonumber
& \geq & \frac{3r_{j}}{2}\\ \nonumber
& \geq & \frac{1}{2}\dist\left( Z_{j}, \partial B_{d_{j}}\left( Y_{j}\right) \right).  
\end{eqnarray} 
We conclude that $B_{\frac{r_{j}}{2}}\left( W_{j}\right) \subset A_{j}$ and
\begin{eqnarray}
\label{}
u\left( Z_{j} \right) \geq \frac{M_{j}}{\dist\left( Z_{j}, \partial B_{d_{j}}\left( Y_{j}\right) \right)}  \geq \frac{\dist \left( X, \partial B_{d_{j}}\left( Y_{j}\right) \right)u\left( X\right)}{\dist\left( Z_{j}, \partial B_{d_{j}}\left( Y_{j}\right) \right)} \geq  \frac{1}{2}u\left( X \right),  
\end{eqnarray}
for all $X \in B_{\frac{r_{j}}{2}}\left( W_{j}\right)$. From above inequality we obtain
\begin{eqnarray}
\label{t 3.1.m}
\sup_{B_{\frac{r_{j}}{2}}\left( W_{j}\right)}u \leq 2u\left( Z_{j} \right).
\end{eqnarray} 
Since $B_{r_{j}}\left( Z_{j}\right) \subset \left\lbrace u > 0 \right\rbrace$, by Harnack inequality, there exist universal constants $C',c' >0$ such that
\begin{eqnarray}
\inf\limits_{B_{\frac{3}{4}r_{j}}\left( Z_{j}\right)} u \geq c'u\left( Z_{j} \right)  - C' r_{j}.
\end{eqnarray}
Therefore, we conclude,
\begin{eqnarray}
\label{t 3.1.n}
\sup\limits_{B_{\frac{1}{4}r_{j}}\left( W_{j}\right)} u \geq c'u\left( Z_{j} \right) - C' r_{j}.
\end{eqnarray} 
For each $j$, consider the normalized function $u_{j}: B_{1}\left( 0 \right)\rightarrow \mathbb{R}$, defined as
\begin{eqnarray}
\label{t 3.1.2}
u_{j}\left( X \right):= \frac{u\left( W_{j} + \frac{r_{j}}{2}X\right)}{u\left( Z_{j} \right)} . 
\end{eqnarray}
Notice that from (\ref{t 3.1.m}), (\ref{t 3.1.j}) and (\ref{t 3.1.n}), we have (for $j$ sufficiently large)
\begin{eqnarray}
\label{t 3.1.p}
\max_{B_{1}\left( 0 \right)}u_{j} \leq 2, \quad \max_{B_{1}\left( 0 \right)}u_{j} \geq \frac{c'}{2}, \quad  u_{j}\left(0 \right) = 0. 
\end{eqnarray}
Let $h$ be the  $\mathcal{A}$-harmonic function in $B_{\frac{1}{2}r_{j}}\left( W_{j}\right)$ taking boundary data equal $u$. By (\ref{th 2.2.2}), as in Theorem \ref{th 2.2}, we have
\begin{eqnarray}
\label{t 3.1.b}
\int\limits_{B_{\frac{1}{2}r_{j}}\left( W_{j}\right)}  \langle \mathcal{A}\left( X, \nabla u \right), \nabla u \rangle - \langle \mathcal{A}\left( X, \nabla h \right), \nabla h \rangle  dX \leq K\int\limits_{B_{\frac{1}{2}r_{j}}\left( W_{j}\right)} \vert u^{m} - h^{m} \vert dX + C r^{n}_{j}.
\end{eqnarray}
Analogously, for each $j$ sufficiently large,   consider the normalized  function $h_{j}: B_{1}\left( 0 \right)\rightarrow \left( 0, 2\right)$, to be
\begin{eqnarray}
\label{t 3.1.2}
h_{j}\left( X \right):= \frac{h\left( W_{j} + \frac{r_{j}}{2}X\right)}{u\left( Z_{j} \right)}. 
\end{eqnarray}
Easily we verify that
\begin{eqnarray}
\label{t 3.1.3}
\left \{
\begin{array}{rcll}
\div \left( \mathcal{A}\left( W_{j} + \frac{r_{j}}{2}X, \nabla h_{j} \right) \right)& = & 0   & \mbox{in} \ B_{1}\left( 0 \right) \\
h_{j} & = & u_{j}  & \mbox{on} \ \partial B_{1}\left( 0 \right). \\ 
\end{array}
\right.
\end{eqnarray} 
and $h_{j}$ is the unique minimizer of
\begin{eqnarray}
\label{t 3.1.8} 
\mathfrak{F}_{j}\left( v \right):= \int_{B_{1}\left( 0\right)}\langle \mathcal{A}\left( W_{j} + \frac{r_{j}}{2}X, \nabla v \right), \nabla v \rangle dX, 
\end{eqnarray}
among functions $v \in W_{0}^{1, p}\left( B_{1}\right) + h_{j}$. Also, from the normalization, 
\begin{eqnarray}
\label{t 3.1.4}
\nabla u_{j}\left( X \right) = \frac{r_{j}}{2u\left( Z_{j} \right) } \nabla  u\left( W_{j} + \dfrac{1}{2}r_{j}X\right), \quad  \nabla h_{j}\left( X \right) = \frac{r_{j}}{2u\left( Z_{j} \right)} \nabla  h\left( W_{j} + \frac{1}{2}r_{j}X\right), 
\end{eqnarray}
for all $X \in B_{1}\left( 0 \right)$. By change of variables and (G3) we obtain
\begin{eqnarray*}
\int\limits_{B_{\frac{1}{2}r_{j}}\left( W_{j}\right)}\vert u^{m} - h^{m} \vert dX &\leq& u^{m}\left( Z_{j} \right)C\left( m \right) \left( \frac{r_{j}}{2}\right)^{n}\int\limits_{B_{1}\left( 0\right)}\vert u_{j} - h_{j} \vert dX \\
&\leq& C\left(n,m\right) u^{m}\left( Z_{j} \right)\left( \frac{r_{j}}{2}\right)^{n}.
\end{eqnarray*} 
Similarly,
\begin{eqnarray*}
\int\limits_{B_{\frac{1}{2}r_{j}}\left( W_{j}\right)} \langle \mathcal{A}\left( X, \nabla u \right), \nabla u \rangle dX = \left(\frac{r_{j}}{2u\left( Z_{j}\right)}\right)^{-p} \left( \frac{r_{j}}{2}\right)^{n}\int\limits_{B_{1}\left( 0\right)}\langle \mathcal{A}\left( W_{j} + \frac{r_{j}}{2}X, \nabla u_{j}\right), \nabla u_{j}\rangle dX.
\end{eqnarray*} 
We conclude therefore that
\begin{eqnarray}
\label{t 3.1.5}
\int_{B_{1}\left( 0\right)}\langle \mathcal{A}\left( W_{j} + \frac{r_{j}}{2}X, \nabla u_{j}\right), \nabla u_{j}\rangle - \langle \mathcal{A}\left( W_{j} + \frac{r_{j}}{2}X, \nabla h_{j}\right), \nabla h_{j}\rangle dX \leq l_{j}\rightarrow 0. 
\end{eqnarray}
where 
\begin{eqnarray}
l_{j}:= Cr^{m}_{j}\left(\frac{r_{j}}{u\left( Z_{j}\right)} \right)^{p-m} = Cu^{m}\left( Z_{j} \right)\left(\frac{r_{j}}{u\left( Z_{j}\right)} \right)^{p} \leq C \left(\frac{r_{j}}{u\left( Z_{j}\right)} \right)^{p}.
\end{eqnarray}
Moreover, since $0 \leq u_{j} \leq 2$ in $B_{1}$, we obtain by Morrey' Theorem (as in the proof of Theorem \ref{th 2.2}) that $u_{j}$ and $h_{j}$ are uniform H\"older continuous in $B_{\frac{8}{9}}\left( 0 \right)$. Thus, up to a subsequence,
\begin{eqnarray*}
\label{t 3.1.6}
u_{j} \rightarrow u_{0}\ \ \mbox{and} \ \ h_{j} \rightarrow h_{0},
\end{eqnarray*}
uniformly in $\overline{B_{\frac{7}{9}}}\left( 0 \right)$ and weakly in $W^{1,p}$. Passing the limit in (\ref{t 3.1.3}) and (\ref{t 3.1.8}),  we find 
\begin{eqnarray}
\label{t 3.1.7}
\left \{
\begin{array}{rcll}
\div \left( \mathcal{A}\left( W_{0}, \nabla h_{0} \right) \right)& = & 0 & \mbox{in} \ B_{1}\left( 0 \right) \\
h_{0} & = & u_{0} & \mbox{on} \ \partial B_{1}\left( 0 \right), \\ 
\end{array}
\right.
\end{eqnarray}
Up to a subsequence, $W_{j} \rightarrow W_{0} \in \partial \{ u>0 \}$, and $h_0$ is the unique minimizer of 
\begin{eqnarray*}
\label{t 3.1.9}
\mathfrak{F}\left( v \right):= \int_{B_{1}\left( 0\right)}\langle \mathcal{A}\left( W_{0}, \nabla v \right), \nabla v \rangle dX.
\end{eqnarray*}
 From (\ref{t 3.1.5}), we find
\begin{eqnarray*}
 u_{0} = h_{0}. 
\end{eqnarray*}
Thus, $u_{0}$ itself solves the elliptic PDE
\begin{eqnarray*}
\div \left( \mathcal{A}\left( W_{0}, \nabla u_{0} \right) \right)& = & 0  \ \ \mbox{in} \ \ B_{1}\left( 0 \right).
\end{eqnarray*} 
Therefore, since $u_{0}\left( 0 \right)=0$ and $u_{0} \geq 0$, we obtain, by the strong maximum principle, $u_{0} \equiv 0$, which contradicts (\ref{t 3.1.p}). Theorem \ref{t 3.1} is proven.
\end{proof}

The optimal regularity estimate on $u$ established in Theorem \ref{t 3.1} implies that $u$ grows at most linearly away from free surface  $\partial \{u> 0 \}$. From energy considerations, we will show next that minimizers do grow precisely at a linear fashion. 

\begin{theorem} 
\label{t 3.2}
Given a subdomain $\Omega' \Subset \Omega$, there exist constants $c, ~ d_0 >0$ that depend only on $\Omega'$ and universal constants, such that if $X \in \{ u > 0 \} \cap \Omega'$ and $\dist (X, \partial \{u > 0 \})\leq d_{0}$, then
$$
	u(X) \ge c \cdot \dist (X, \partial \{u > 0 \}). 
$$ 
\end{theorem}

\begin{proof}
Given a point $X_{0} \in \{ u > 0 \} \cap \Omega'$ we denote $d:= \dist( X_{0}, \partial \{u > 0 \})$. Define
\begin{eqnarray*}
v\left( X \right) = \frac{u\left( X_{0} + dX\right)}{d}, \ \forall X \in B_{1}\left( 0 \right).  
\end{eqnarray*}
Notice that
\begin{eqnarray*}
\div \left( \mathcal{A}\left( X_0 + dX, \nabla v \right) \right) = d^{m}mf \left( X_0 + dX \right)v^{m-1} \ \ \mbox{in} \ \ \left\lbrace v > 0 \right\rbrace,   
\end{eqnarray*}
where $\mathcal{A}\left( Y, \xi \right)$ is as in Theorem \ref{th 2.2}. Let $Y_{0} \in \partial \left\lbrace u > 0 \right\rbrace$ such that $d = \vert X_{0} - Y_{0} \vert$. Since $u$ is Lipschitz continuous 
\begin{eqnarray}
\label{v bounded}
\Vert v \Vert_{L^{\infty}\left( B_{1}\left( 0\right)\right)} \leq 2.
\end{eqnarray}
By Harnack inequality,  
\begin{eqnarray*}
v\left( X \right) \leq Cv\left( 0 \right) + d^{m}C\Vert f \Vert_{L^{\infty}\left( B_{1}\left( 0 \right)\right) }, \quad \forall X \in B_{\frac{3}{4}}\left( 0 \right).    
\end{eqnarray*}
Let $\psi$ be a nonnegative, smooth cut-off function such that 
\begin{eqnarray}
\psi = \left \{
\begin{array}{lll}
0, \ &\mbox{if}& \ X \in B_{\frac{1}{10}}\left( 0 \right), \\ 
1, \ &\mbox{if}& \ B_{1}\left( 0 \right)\setminus B_{\frac{1}{2}}\left( 0\right) .\\ 
\end{array}
\right.
\end{eqnarray}
Define the test function $\xi$ in $B_{1}\left( 0 \right)$ by
\begin{eqnarray}
\xi := \left \{
\begin{array}{ccr}
\min \left\lbrace v, \left( Cv\left( 0 \right) + d^{m}C\Vert f \Vert_{L^{\infty}\left( B_{1}\left( 0 \right)\right) }\right)\psi \right\rbrace  &\mbox{in}& \ B_{\frac{3}{4}}\left(0 \right), \\
v \ &\mbox{in}& \ B^{C}_{\frac{3}{4}}\left(0 \right).\\
\end{array}
\right.     
\end{eqnarray}
Clearly, 
\begin{eqnarray}
0 \leq \xi \leq v \leq 2.
\end{eqnarray}
By minimality of $v$ in $B_{1}\left( 0 \right)$ we have
\begin{eqnarray}
\label{t 3.2.1}
\int\limits_{\Pi}  G(X_{0}+ dX, \nabla \xi ) - G(X_{0} + dX, \nabla v) dX & = & \nonumber \int\limits_{B_{1}\left( 0 \right)}  G(X_{0}+ dX, \nabla \xi ) - G(X_{0} + dX, \nabla v) dX \\ \nonumber
& \geq &  \int\limits_{B_{1}\left( 0 \right)}  g\left( X_{0}+ dX,  dv \right) - g\left( X_{0}+ dX, d\xi \right) dX  \\ 
& = & d^{m}\int\limits_{\Pi}  f\left(X_{0}+ dX \right)\left(  v^{m} - \xi^{m} \right) dX \\ \nonumber
& + & \int\limits_{B_{1}\left( 0 \right)} Q\left( X_{0}+ dX \right)\left( \chi_{\left\lbrace v > 0 \right\rbrace} - \chi_{\left\lbrace \xi > 0\right\rbrace }\right) dX, 
\end{eqnarray} 
where $\Pi:=\left\lbrace \left( Cv\left( 0 \right) + d^{m}C\Vert f \Vert_{L^{\infty}\left( B_{1}\left( 0 \right)\right) }\right)\psi < v \right\rbrace$. In addition, we estimate,
\begin{eqnarray}
\label{t 3.2.a}
d^{m}\int_{\Pi} \vert f\left(X_{0}+ dX \right)\left(  v^{m} - \xi^{m} \right)\vert dX 
& \leq & d^{m}Km2^{m-1}\int_{\Pi}  \vert  v - \xi \vert dX \\ \nonumber
& \leq & d^{m}C\int_{\Pi}  \vert  v - \xi \vert^{p} dX + d^{m}C \\ \nonumber
& \leq & d^{m}C\int_{\Pi}  \vert \nabla \left( v - \xi \right) \vert^{p} dX + d^{m}C \\ \nonumber
& \leq & d^{m}C\int_{\Pi}  \vert \nabla  v \vert^{p} dX + d^{m}C\int_{\Pi} \vert \nabla\xi \vert^{p} dX + d^{m}C
\end{eqnarray}
and 
\begin{eqnarray}
\label{t 3.2.b}
\int_{B_{1}\left( 0 \right)} Q\left( X_{0}+ dX \right)\left( \chi_{\left\lbrace v > 0 \right\rbrace} - \chi_{\left\lbrace \xi > 0\right\rbrace }\right) dX \nonumber
& = & \int_{B_{1}\left( 0 \right)} Q\left( X_{0}+ dX \right)\left( 1 - \chi_{\left\lbrace \xi > 0\right\rbrace }\right) dX \\ \nonumber
& = &\int_{B_{1}\left( 0 \right)} Q\left( X_{0}+ dX \right)\chi_{\left\lbrace \xi = 0 \right\rbrace } dX \\ \nonumber
& \geq & \epsilon \Leb \left( \left\lbrace \xi = 0 \right\rbrace \right) \\ 
& \geq & \epsilon \Leb \left( B_{\frac{1}{10}}\left( 0 \right) \right).
\end{eqnarray}
with $C> 0$ being a universal constant. From condition (G2) and (G3) we can further estimate 
\begin{eqnarray}
\label{t 3.2.2}
p G( X_{0}+ dX, \nabla v) & = & \nonumber \langle \mathcal{A}\left(  X_{0}+ dX,  \nabla v \right), \nabla v \rangle \\
& \geq & \lambda \vert \nabla v \vert^{p}
\end{eqnarray}
and
\begin{eqnarray}
\label{t 3.2.3}
p G( X_{0}+ dX, \nabla \xi) & = & \nonumber \langle \mathcal{A}\left(  X_{0}+ dX,  \nabla \xi \right), \nabla \xi \rangle \\ 
& \leq & \lambda^{-1} \vert \nabla \xi \vert^{p}.
\end{eqnarray}
We, therefore, have, 
\begin{eqnarray}
\label{t 3.2.4}
\int\limits_{\Pi}  G(X_{0}+ dX, \nabla \xi ) - G(X_{0} + dX, \nabla v) dX & \leq & \ \nonumber \frac{\lambda^{-1}}{p} 
\int\limits_{\Pi}\vert \left( Cv\left( 0 \right) + d^{m}C\Vert f \Vert_{L^{\infty}\left( B_{1}\left( 0 \right)\right) }\right) \nabla \psi \vert^{p} dX \\ \nonumber
&-& \frac{\lambda}{p} \int\limits_{\Pi}\vert \nabla v \vert^{p} dX \\ \nonumber
& \leq & C\int\limits_{B_{1}\left( 0 \right)} \left( v^{p}\left( 0 \right) + d^{mp}\Vert f \Vert^{p}_{L^{\infty}\left( B_{1}\left( 0 \right)\right) }\right) \vert \nabla \psi \vert^{p} dX  \\ \nonumber
&-& \frac{\lambda}{p}  \int\limits_{\Pi}\vert \nabla v \vert^{p} dX \\ 
&\leq & Cv^{p}\left( 0 \right) + d^{mp}C - \frac{\lambda}{p}  \int_{\Pi}\vert \nabla v \vert^{p} dX,
\end{eqnarray}
where, as before, $C>0$ is a universal constant. From inequalities (\ref{t 3.2.a}), (\ref{t 3.2.b}), (\ref{t 3.2.4}) together with (\ref{t 3.2.1}) and taking $d$ universally small, we obtain
\begin{eqnarray}
Cv^{p}\left( 0 \right) & \geq & \nonumber   
\left( \frac{\lambda}{p} - d^{m}C \right) \int_{\Pi} \vert \nabla v \vert^{p} dX + c \Leb \left( B_{\frac{1}{10}}\left( 0 \right) \right) \\ \nonumber
& \geq & c \Leb \left( B_{\frac{1}{10}}\left( 0 \right) \right).
\end{eqnarray}
with $c>0$ and $C>0$ universal constants. In conclusion,
\begin{eqnarray*}
v\left( 0 \right) \geq c,
\end{eqnarray*} 
and Theorem \ref{t 3.2} is proven.
\end{proof}

By a refinement of the arguments in the proof of Theorem \ref{t 3.2}, we also obtain geometric strong nondegeneracy property of $u$. More precisely, we show
\begin{theorem} 
\label{t 3.3}
Given a subdomain $\Omega' \Subset \Omega$, there exist constants $c_1>0$ and $r_{1}>0$ that depend only on $\Omega'$ and universal constants, such that if $X_0 \in \partial \{ u > 0 \} \cap \Omega'$, $0 < r \leq r_{1}$, then
$$
	\sup\limits_{X\in B_r(X_0)} u(X) \ge c_1 r. 
$$ 
\end{theorem}

\begin{proof}
The proof follows the reasoning from Theorem \ref{t 3.2}. We include the details as a courtesy to the readers. Given a point $X_{0} \in \partial \{ u > 0 \} \cap \Omega'$, define 
\begin{eqnarray*}
v\left( X \right) = \frac{u\left( X_{0} + rX\right)}{r},  \quad  \forall X \in B_{1}\left( 0 \right).  
\end{eqnarray*} 
Let $h_{r}$ be the universal barrier given by
\begin{eqnarray}
\left \{
\begin{array}{rcl}
\div \left( \mathcal{A}\left( X_{0} + rX, \nabla h_{r} \right) \right)& = & 0  \ \mbox{in} \ B_{1}\left( 0 \right)\setminus  B_{\frac{1}{2}}\left( 0 \right)   \\
h_{r} & = & 1 \ \mbox{on} \ \partial B_{1}\left( 0 \right), \\ 
h_{r} & = & 0 \ \mbox{in} \ \overline{B_{\frac{1}{2}}\left( 0 \right)}.\\ 
\end{array}
\right.
\end{eqnarray}
Define the test function $\xi$ in $B_{1}\left( 0 \right)$ by
\begin{eqnarray}
\xi \left( X \right)  = \min \left\lbrace v\left( X \right), h_{r}\left( X \right) \sup\limits_{B_{1}(0)}v \right\rbrace.
\end{eqnarray}
As in Theorem \ref{t 3.2} we have 
\begin{eqnarray}
0 \leq \xi \leq v \leq 2.
\end{eqnarray}
By minimality of $v$ in $B_{1}\left( 0 \right)$ we can estimate
\begin{eqnarray}
\label{t 3.3.1}
\int_{\Pi}  G(X_{0} + rX, \nabla \xi ) - G(X_{0} + rX, \nabla v) dX & = & \nonumber \int_{B_{1}\left( 0 \right)}  G(X_{0} + rX, \nabla \xi ) - G(X_{0} + rX, \nabla v) dX \\ \nonumber
& \geq &  \int_{B_{1}\left( 0 \right)}  g\left( X_{0} + rX,  rv \right) - g\left( X_{0} + rX, r\xi \right) dX  \\ 
& = & r^{m}\int_{\Pi}  f\left(X_{0} + rX \right)\left(  v^{m} - \xi^{m} \right) dX \\ \nonumber
& + & \int_{B_{1}\left( 0 \right)} Q\left( X_{0} + rX \right)\left( \chi_{\left\lbrace v > 0 \right\rbrace} - \chi_{\left\lbrace \xi > 0\right\rbrace }\right) dX, 
\end{eqnarray}
where $\Pi:=\left\lbrace h_{r}\sup\limits_{B_1(X_0)}u < v \right\rbrace$. Also we have
\begin{eqnarray}
\label{t 3.3.2}
r^{m}\int_{\Pi} \vert f\left(X_{0}+ rX \right)\left(  v^{m} - \xi^{m} \right)\vert dX 
& \leq & r^{m}Km2^{m-1}\int_{\Pi}  \vert  v - \xi \vert dX \\ \nonumber
& \leq & r^{m}C\int_{\Pi}  \vert  v - \xi \vert^{p} dX + r^{m}C \\ \nonumber
& \leq & r^{m}C\int_{\Pi}  \vert \nabla \left( v - \xi \right) \vert^{p} dX + r^{m}C \\ \nonumber
& \leq & r^{m}C\int_{\Pi}  \vert \nabla  v \vert^{p} dX + r^{m}C\int_{\Pi} \vert \nabla\xi \vert^{p} dX + r^{m}C.
\end{eqnarray}
Similarly we estimate
\begin{eqnarray}
\label{t 3.3.3}
\int_{B_{1}\left( 0 \right)} Q\left( X_{0}+ rX \right)\left( \chi_{\left\lbrace v > 0 \right\rbrace} - \chi_{\left\lbrace \xi > 0\right\rbrace }\right) dX \nonumber
& = & \int_{B_{1}\left( 0 \right)} Q\left( X_{0}+ rX \right)\left( 1 - \chi_{\left\lbrace \xi > 0\right\rbrace }\right) dX \\ \nonumber
& = &\int_{B_{1}\left( 0 \right)} Q\left( X_{0}+ rX \right)\chi_{\left\lbrace \xi = 0 \right\rbrace } dX \\ \nonumber
& \geq & \epsilon \Leb \left( \left\lbrace \xi = 0 \right\rbrace \right) \\ 
& \geq & \epsilon \Leb \left( B_{\frac{1}{2}}\left( 0 \right) \right),
\end{eqnarray}
where $C> 0$ is a universal constant.  Taking into account (G2) and (G3), we obtain
\begin{eqnarray}
\label{t 3.3.6}
\int_{\Pi}  G(X_{0} + rX, \nabla \xi ) - G(X_{0} + rX, \nabla v) dX & \leq & \ \nonumber \frac{\lambda^{-1}}{p} \int_{\Pi} \vert \sup\limits_{B_1(0)}v  \nabla h_{r} \vert^{p} dX - \frac{\lambda}{p}  \int_{\Pi}\vert \nabla v \vert^{p} dX \\ \nonumber
& \leq & C\left( \sup\limits_{B_1(0)}v \right)^{p} \int_{B_{1}\left( 0 \right)} \vert \nabla h_{r}\vert^{p} dX  \\ 
&-& \frac{\lambda}{p}  \int_{\Pi}\vert \nabla v \vert^{p} dX,
\end{eqnarray}
for $C>0$  universal. Also by $C^{1,\alpha}$ estimates for $h_{r}$, we have
\begin{eqnarray*}
\label{t 3.3.7}
C\left( \sup\limits_{B_1(0)}v \right)^{p} \int_{B_{1}\left( 0 \right)} \vert \nabla h_{r}\vert^{p} dX \leq C\left( \sup\limits_{B_1(0)}v \right)^{p}
\end{eqnarray*}
for a universal constant $C>0$. Plugging inequalities (\ref{t 3.3.2}), (\ref{t 3.3.3}) and (\ref{t 3.3.6}) inside (\ref{t 3.3.1}) and taking $r$ sufficiently small, we obtain
\begin{eqnarray}
C \left( \sup\limits_{B_1(0)}v \right)^{p} & \geq & \nonumber   
\left( \frac{\lambda}{p} - r^{m}C \right) \int_{\Pi} \vert \nabla v \vert^{p} dX + c\Leb \left( B_{\frac{1}{2}}\left( 0 \right) \right) \\ \nonumber
& \geq & c \Leb \left( B_{\frac{1}{2}}\left( 0 \right) \right).
\end{eqnarray}
with $c>$ and $C>0$ universal constants. Therefore, $\sup\limits_{B_1(0)}v \geq c$,  and strong nondegeneracy is proven. 
\end{proof}

\section{Hausdorff estimates of the free boundary} \label{S Hausdorff Est}

In this section we turn out attention to fine Hausdorff estimates on the free surface $\partial \{u > 0 \}$. 

\begin{theorem}
\label{t 4.1}
Given a subdomain $\Omega' \Subset \Omega$ and $Z\in \partial \{u > 0 \} \cap \Omega'$, there exist constants $r_{0} > 0$ and $0<\varsigma<1$ that depend  only on $\Omega'$ and universal constants, such that,
\begin{equation}\label{uniform positive density}
   \varsigma \omega_n r^n  \le \Leb \left ( B_r(Z) \cap
    \{ u > 0 \} \right ) \le (1 - \varsigma) \omega_n r^n, 
\end{equation}
for all $0 \leq r \leq r_{0}$.
\end{theorem}

\begin{proof}
Let $X_{0} \in \bar{B}_{\frac{r}{4}}\left( Z \right)$ be a maximum point of $u$, i.e.,  
\begin{eqnarray*}
u(X_{0}) = \sup\limits_{X\in B_{\frac{r}{4}}(Z)} u(X). 
\end{eqnarray*}
By strong nondegeneracy property ($0< r \leq r_{0}$ with $r_{0}=r_{1}$ and $r_{1}>0$ as in Theorem \ref{t 3.3}) we obtain
\begin{eqnarray*}
u(X_{0}) \geq cr > 0, 
\end{eqnarray*}
where $c>0$ is a universal constant. By Lipschitz continuity of $u$ there exists a universal constant $C>0$ such that
\begin{eqnarray*}
u(X)\geq u(X_{0}) - C\vert X - X_{0} \vert \geq \frac{cr}{2} > 0, \ \ \forall X \in B_{\frac{c}{2C}r}\left( X_{0} \right),  
\end{eqnarray*}
with $C \geq c$.
Hence,
\begin{eqnarray*}
B_{\frac{c}{2C}r}\left( X_{0} \right) \subset B_r(Z) \cap \{ u > 0 \},  
\end{eqnarray*}
and the estimate by below in (\ref{uniform positive density}) follows. 
\par
Let us now prove estimate by above. We argue by contradiction, i.e., let us assume that there exists a sequence of positive real numbers $r_{j}$ with $r_{j}\searrow 0$ as $j \rightarrow \infty$ and 
\begin{eqnarray}
\label{t 4.1.1} 
\frac{\Leb \left ( B_{r_{j}}(Z) \cap \{ u = 0 \} \right )}{r^{n}_{j}}\rightarrow 0. 
\end{eqnarray}
We define the sequence $u_{j}: B_{1}\left( 0 \right)\rightarrow \mathbb{R}$ by
\begin{eqnarray}
\label{t 4.1.b}
u_{j}\left( X \right):= \frac{u\left( Z + r_{j}X\right)}{r_{j}}. 
\end{eqnarray}
Let $h_{j}$ be the solution to 
\begin{eqnarray}
\label{t 4.1.a}
\left \{
\begin{array}{rcll}
\div \left( \mathcal{A}\left( Z + r_{j}X, \nabla h_{j} \right) \right)& = & 0  &  \mbox{in} \ B_{1}\left( 0 \right) \\
h_{j} & = & u_{j} & \mbox{on} \ \partial B_{1}\left( 0 \right). \\ 
\end{array}
\right.
\end{eqnarray}
Notice that by Lipschitz continuity of $u$,  both $u_{j}$ and  $h_{j}$ are bounded. Estimates (\ref{th 2.2.2}) and (\ref{th 2.2.4}) from the proof of Theorem \ref{th 2.2} gives after renormalization,  
\begin{eqnarray}
\label{t 4.1.c}
\int_{B_{1}\left( 0 \right)}\vert \nabla \left( h_{j} - u_{j}\right)\vert^{p} dX & \leq & \nonumber r^{m}_{j}\int_{B_{1}\left( 0 \right)}\vert h^{m}_{j} - u^{m}_{j} \vert dX + C \dfrac{\Leb \left ( B_{r_{j}}(Z) \cap \{ u = 0 \} \right )}{r^{n}_{j}} \\ \nonumber
& \leq & r^{m}_{j}C\int_{B_{1}\left( 0 \right)}\vert h_{j} - u_{j} \vert dX + C \dfrac{\Leb \left ( B_{r_{j}}(Z) \cap \{ u = 0 \} \right )}{r^{n}_{j}}\\ \nonumber
& \leq & r^{m}_{j}C\int_{B_{1}\left( 0 \right)}\vert h_{j} - u_{j} \vert^{p} dX + r^{m}_{j}C + C \dfrac{\Leb \left ( B_{r_{j}}(Z) \cap \{ u = 0 \} \right )}{r^{n}_{j}} \\ \nonumber
&\leq & r^{m}_{j}C\int_{B_{1}\left( 0 \right)}\vert \nabla \left( h_{j} - u_{j}\right) \vert^{p} dX + r^{m}_{j}C + C\dfrac{\Leb \left ( B_{r_{j}}(Z) \cap \{ u = 0 \} \right )}{r^{n}_{j}},  
\end{eqnarray} 
where $C>0$ is a universal constant. Hence, for $j$ sufficietly large we have
\begin{eqnarray}
\label{t 4.1.2}
\int_{B_{1}\left( 0 \right)}\vert \nabla \left( h_{j} - u_{j}\right)\left( X \right) \vert^{p} dX & \leq &  r^{m}_{j}C + C \dfrac{\Leb \left ( B_{r_{j}}(Z) \cap \{ u = 0 \} \right )}{r^{n}_{j}}. 
\end{eqnarray} 
Moreover, by Lipschitz regularity $u$ and $C^{1, \alpha}$ elliptic estimate we may assume that 
\begin{eqnarray*}
u_{j} \rightarrow u_{0}\ \ \mbox{and} \ \ h_{j} \rightarrow h_{0}
\end{eqnarray*}
uniformly in $B_{\frac{4}{5}}\left( 0 \right)$. Since $h_{j}$ is the solution to problem (\ref{t 4.1.a}) we obtain 
\begin{eqnarray*}
\div \left( \mathcal{A}\left( Z , \nabla h_{0}\left( Y\right)  \right) \right) = 0, \ \ \mbox{in} \ B_{\frac{1}{2}}\left( 0 \right).  
\end{eqnarray*}
Also follows from (\ref{t 4.1.1}) and (\ref{t 4.1.2}) that
\begin{eqnarray*}
u_{0} = h_{0} + c, \ \ \mbox{in} \ B_{\frac{1}{2}}\left( 0 \right).  
\end{eqnarray*}
where $c$ is a constant. Then,
\begin{eqnarray*}
\div \left( \mathcal{A}\left( Z , \nabla u_{0}\left( Y\right)  \right) \right) = 0, \ \ \mbox{in} \ B_{\frac{1}{2}}\left( 0 \right).  
\end{eqnarray*}
Since $u_{0} \geq 0$ and $u_{0}\left( 0 \right)=0$, by the strong maximum principle, we obtain $u_{0} \equiv 0$ in $B_{\frac{1}{2}}\left( 0 \right)$. But this contradicts strong nondegeneracy property, granted by Theorem \ref{t 3.3}.  
\end{proof}

Theorem \ref{t 2.3} reveals the Euler-Lagrange equation $u$ satisfies within its set of positivity. To further investigate the behavior of $u$ along the free boundary, we need to obtain the equation $u$ satisfies through the free surface of discontinuity of the functional, $\partial \{u > 0 \}$.

\begin{lemma}
\label{l 2}
Let $u$ be a minimizer to the functional \eqref{Funct Intro}, then
\begin{eqnarray*}
	\div  \left ( \mathcal{A}\left( X, \nabla u \right) \right) - mf(X)u^{m-1} \chi_{\{u > 0\}} \geq 0, \ \ \mbox{in} \ \  \Omega
\end{eqnarray*}	
in the sense of distribution. In particular it defines a Radon measure
\begin{eqnarray*}
\Lambda:= \div  \left ( \mathcal{A}\left( X, \nabla u \right) \right) - mf(X)u^{m-1} \chi_{\{u > 0\}}
\end{eqnarray*}	
Furthermore, the support of $\Lambda$ is contained $\partial \left\lbrace u > 0 \right\rbrace$.
\end{lemma}
\begin{proof}
Let $\zeta \in C_{0}^{\infty}\left( \Omega \right)$ be a nonnegative function. Given $\varepsilon >0$, by minimality of $u$, we have
\begin{eqnarray*}
0 \nonumber &\leq & \frac{1}{\varepsilon}\int_{\Omega} G(X, \nabla \left( u - \varepsilon \zeta\right) ) - G(X, \nabla u) dX  + \frac{1}{\varepsilon}\int_{\Omega}  f\left( X\right)\left[ \left( \left( u - \varepsilon \zeta \right)^{+}\right)^{m} - u^{m} \right]dX \\ \nonumber & + & \frac{1}{\varepsilon}\int_{\Omega} Q\left( \chi_{\left\lbrace u - \varepsilon \zeta > 0 \right\rbrace} - \chi_{\left\lbrace u > 0\right\rbrace }\right) dX\\ 
& \leq & - \int_{\Omega} \mathcal{A}\left(X, \nabla \left( u - \varepsilon \zeta\right) \right) \cdot \nabla \zeta dX + \frac{1}{\varepsilon}\int_{\Omega}  f\left( X\right)\left[ \left( \left( u - \varepsilon \zeta \right)^{+}\right)^{m} - u^{m} \right]dX. \\ \nonumber
\end{eqnarray*} 
Also,
\begin{eqnarray*}
\frac{1}{\varepsilon}\int_{\Omega}  f\left( X\right)\left[ \left( \left( u - \varepsilon \zeta \right)^{+}\right)^{m} - u^{m} \right]dX 
& = & \frac{1}{\varepsilon} \int_{\left\lbrace u > 0 \right\rbrace}  f\left( X\right)\left[ \left( \left( u - \varepsilon \zeta \right)^{+}\right)^{m} - u^{m} \right]dX \\ \nonumber
& + & \frac{1}{\varepsilon}\int_{\left\lbrace u = 0 \right\rbrace}  f\left( X\right)\left[ \left( - \varepsilon \zeta \right)^{+} \right]^{m}dX \\ \nonumber
& = & \int_{\left\lbrace u > 0 \right\rbrace}  f\left( X\right) \frac{\left[ \left( \left( u - \varepsilon \zeta \right)^{+}\right)^{m} - u^{m} \right]}{\varepsilon}dX.
\end{eqnarray*} 
Taking $\varepsilon \rightarrow 0$ we obtain
\begin{eqnarray*}
- \int_{\Omega} \mathcal{A}\left(X, \nabla u \right) \cdot \nabla \zeta dX - \int_{\Omega}\left( mf\left( X \right)u^{m-1}\chi_{\left\lbrace u>0 \right\rbrace}\right) \zeta dX \geq 0.
\end{eqnarray*}
Moreover, as in Theorem \ref{t 2.3}, 
\begin{eqnarray}
\label{r 1}
\div \mathcal{A}\left(X, \nabla u \right) = mf\left( X \right)u^{m-1} \ \mbox{in} \ \left\lbrace u>0 \right\rbrace. 
\end{eqnarray}
Hence, the measure $\Lambda$ defined by
\begin{eqnarray*}
\int_{\Omega}\zeta d\Lambda:= - \int_{\Omega} \mathcal{A}\left(X, \nabla u \right) \cdot \nabla \zeta dX - \int_{\Omega} \left( mf\left( X \right)u^{m-1}\chi_{\left\lbrace u>0 \right\rbrace}\right)\zeta dX.
\end{eqnarray*} 
is a nonnegative Radon measure with support in $\Omega \cap \partial \{u > 0 \}$. 
\end{proof}

\medskip 

With the aid of the measure $\Lambda$, we can establish fine upper and lower control on the $\mathcal{H}^{n-1}$ Hausdorff measure of the free boundary, which ultimately reveal important geometric-measure information on $\partial \{u > 0 \}$.
\begin{theorem} 
\label{t 4.2}
The set $\{ u > 0  \}$ has locally finite
perimeter and for universal constants $\underline{c}$,
$\overline{C}$,  there holds
\begin{equation}\label{estimate on Hausdorff measure}
    \underline{c} r^{n-1} \le \mathcal{H}^{n-1} \left (\partial
     \{u > 0 \} \cap B_r(Z) \right ) \le \overline{C}
    r^{n-1}
\end{equation}
for any ball $B_r(Z)$ centered at a free boundary point, $Z \in \partial \{u > 0 \}$. In
particular, 
$$
	\mathcal{H}^{n-1} \left ( \partial  \{u > 0 \}  \setminus \redbdry   \{u > 0 \}  \right ) = 0.
$$
\end{theorem}

\begin{proof}
Through a  suitable approximation scheme,  $0 \leq \zeta_k\leq 1$, $\zeta_k$ test function with $\zeta_k \to \chi_{B_{r}\left( Z \right)}$,  we have (for almost $r >0$)
\begin{eqnarray}
\label{t 4.2.1}
\int_{B_{r}\left( Z \right)}d\Lambda & = & \nonumber  \int_{\partial B_{r}\left( Z \right)} \mathcal{A}\left(X, \nabla u \right) \cdot \nu dS_{X} - \int_{\Omega}mf\left( X \right)u^{m-1}\chi_{\left\lbrace u>0 \right\rbrace}\zeta dX \\ \nonumber  
& \leq & C \Vert \nabla u \Vert^{p-1}_{L\infty} r^{n-1} + C r^{n-1} \\ 
& \leq & Cr^{n-1},  
\end{eqnarray}
where $C>0$ is a universal constant.  The upper inequality in (\ref{estimate on Hausdorff measure}) is verified. 

To check the lower estimate, let us assume by contradiction that there exists a sequence of positive real numbers $r_{j}$ with $r_{j}\searrow 0$ as $j \rightarrow \infty$ and 
\begin{eqnarray}
\label{t 4.2.2} 
\frac{\mathcal{H}^{n-1} \left (\partial
     \{u > 0 \} \cap B_{r_{j}}(Z) \right )}{r^{n-1}_{j}}\rightarrow 0. 
\end{eqnarray}
With the notation used in Theorem \ref{t 4.1} and Lemma \ref{l 2}, we obtain the sequence of nonnegative measures $\Lambda_{j}$ in $B_{\frac{4}{5}}\left( 0 \right)$, defined by 
\begin{eqnarray}
\label{t 4.2.3}
\Lambda_{j}:= \left[ \div\left( \mathcal{A}\left(Z + r_{j}X, \nabla u_{j} \right)\right) - r^{m}_{j}mf\left( Z + r_{j}X \right)\chi_{\left\lbrace u_{j}>0 \right\rbrace } u^{m-1}_{j}\right] dX.  
\end{eqnarray}
By compactness we can assume that $\Lambda_{j} \rightharpoonup \Lambda_{0}$ in the sense of measures. Moreover, using (\ref{t 4.2.2}) we have
\begin{eqnarray}
\label{t 4.2.4}
\Lambda_{j} \rightharpoonup 0.  
\end{eqnarray}
In the sequel we will show that
\begin{eqnarray}
\label{t 4.2.5}
\Lambda_{0}:= \div \left( \mathcal{A}\left(Z , \nabla u_{0} \right)\right)dX.  
\end{eqnarray}
From the uniform positive density property we know $\Leb \left( \partial \{ u_{0} > 0 \}\right) = 0$. Thus,  we  only need to verify (\ref{t 4.2.5}) for balls $B$ entirely contained in $\{u_{0} = 0 \}$ and in $\{u_{0} > 0 \}$. 
Let $B \subset \{u_{0} > 0 \}$. Define 
\begin{eqnarray*}
\label{}
A_{j}\left( X \right):= \mathcal{A}\left(Z + r_{j}X, \nabla u_{j} \right), \ \ \forall X \in B_{\frac{4}{5}}\left( 0 \right).    
\end{eqnarray*} 
By Lipschitz regularity of $u$ and (G2) we have
\begin{eqnarray*}
\label{}
\vert A_{j} \vert \leq C\left(n,\Lambda \right)\vert \nabla u_{j} \vert^{p-1}\leq C.  
\end{eqnarray*}
Thus, we may extract a subsequence (we will denote by $A_{j}$) such that  
\begin{eqnarray*}
\label{}
A_{j} \rightarrow A_{0} \ \mbox{weak-}\star \mbox{ in } \ L^{\infty} (B_{\frac{4}{5}}). 
\end{eqnarray*} 
Furthermore, $u_{j}$ converges in the $C^{1,\alpha}$ topology to $u_{0}$ in $B$ (see Remark \ref{r 3}). Hence, 
\begin{eqnarray*}
\label{}
 A_{j} \rightarrow \mathcal{A}\left(Z , \nabla u_{0} \right) \  \mbox{weak-}\star \mbox{ in }  \ L^{\infty}\left( B \right).  
\end{eqnarray*}
Also we have
\begin{eqnarray}
\label{t 4.2.9}
\Big | \int_{B}r^{m}_{j}mf\left( Z + r_{j}X \right)\chi_{\left\lbrace u_{j}>0 \right\rbrace}u^{m-1}_{j} dX \Big | \leq r^{m}_{j}C \Leb \left( B\right) \rightarrow 0.  
\end{eqnarray}
Hence, for $B \subset \{u_{0} > 0 \}$, indeed  (\ref{t 4.2.5}) does hold. Now suppose $B \subset \{u_{0} = 0 \}$. Clearly,
\begin{eqnarray*}
\label{}
\left[ \div \left( \mathcal{A}\left(Z , \nabla u_{0} \right)\right)dX \right]  \left( B \right) = 0.  
\end{eqnarray*}
On the other hand, if $B_{k}$ is a sequence of balls such that $B_{k}\nearrow B$ then for some $j_{k} \in \mathbb{N}$ we have 
\begin{eqnarray}
\label{th 4.2.10}
u_{j} \equiv 0 \ \mbox{in} \ B_{k} \ \mbox{for all} \ j > j_{k}.
\end{eqnarray}
Indeed, let $\tilde{B} \subset B$. If there were a subsequence $u_{j_{k}}$ satisfying $u_{j_{k}}\neq 0$ in $\tilde{B}$ then, by strong nondegenracy property (Theorem \ref{t 3.3}), there should exist points $P_{k_{j}} \in \tilde{B}$ such that
\begin{eqnarray}
\label{}
u_{j_{k}}\left( P_{k_{j}}\right) \geq c > 0.  
\end{eqnarray}
Passing to another subsequence we can assume $P_{k_{j}} \rightarrow P \in \tilde{B}$. Since $u_{j_{k}} \rightarrow u_{0} $ uniformly we obtain $u_{0}\left( P \right) > 0$ which is a contradiction.\\
Thus, from (\ref{th 4.2.10}) we obtain 
\begin{eqnarray}
\label{th 4.2.11}
 \Lambda_{j}\left( B \right) \rightarrow 0.  
\end{eqnarray}
Therefore, (\ref{t 4.2.5})  holds for any $B \subset \{u_{0} = 0 \}$, and  combining (\ref{t 4.2.4}) and (\ref{t 4.2.5}) we find 
\begin{eqnarray}
\label{}
\div\left( \mathcal{A}\left( Z, \nabla u_{0} \right) \right) = 0 \ \ \mbox{in} \ \ B_{\frac{4}{5}\left( 0 \right)}.  
\end{eqnarray}
However, as before, this drives us to a contradiction on the nondegeneracy property of $u_0$, Theorem \ref{t 3.3}.
\end{proof} 

\begin{theorem} 
\label{t 5.2}
Let $u$ be a minimizer of (\ref{Funct Intro}). Then
\begin{eqnarray*}
	\div  \left ( \mathcal{A}\left( X, \nabla u \right) \right) -  mf(X)u^{m-1} \chi_{\{u > 0\}} = Q \lfloor \redbdry \{u > 0 \},
\end{eqnarray*}	
in the sense of measures.
\end{theorem}

\begin{proof}
Using Theorem \ref{t 4.2} and standard Hadamard's domain variation type of argument, we obtain the result.
\end{proof}
\vspace{0,5cm}
\begin{remark}
\label{r 3}
Let $u$ be a minimizer of (\ref{Funct Intro}) in $\Omega$ and $B_{r_{j}}\left( X_{j} \right) \subset \Omega$ be a sequence of balls with $r_{j}\rightarrow 0$, $X_{j}\rightarrow X_{0} \in \Omega$, and $u\left( X_{j}\right)= 0$. Consider the sequence blow-up
\begin{eqnarray*}
u_{j}\left( X \right) = \frac{1}{r_{j}}u\left( X_{j} + r_{j}X\right). 
\end{eqnarray*}
Since $u_{j}$ are uniformly Lipschitz continuous, for a subsequence,
\begin{eqnarray}
\label{1}
u_{j}\rightarrow u_{0} \ \mbox{in} \ C_{loc}^{\alpha}\left( \mathbb{R}^{n}\right) \mbox{ for every} \ 0< \alpha <1, 
\end{eqnarray}
\begin{eqnarray}
\label{2}
\nabla u_{j}\rightarrow \nabla u_{0} \ \mbox{weak-}\star  \mbox{ in}  \ L_{loc}^{\infty}\left( \mathbb{R}^{n}\right), 
\end{eqnarray}
\begin{eqnarray}
\label{3}
\partial \left\lbrace u_{j} > 0\right\rbrace \rightarrow \partial \left\lbrace u_{0} > 0\right\rbrace  \ \mbox{locally in the Hausdorff distance}, 
\end{eqnarray}
\begin{eqnarray}
\label{4}
\chi_{\left\lbrace u_{j} > 0\right\rbrace} \rightarrow \chi_{\left\lbrace u_{0} > 0\right\rbrace}  \ \mbox{in} \ L_{loc}^{1}\left( \mathbb{R}^{n}\right). 
\end{eqnarray}
Moreover, by classical truncation argument, see for intance, \cite{BF},  
\begin{eqnarray}
\label{5}
\nabla u_{j}\rightarrow \nabla u_{0} \ \mbox{a.e.} 
\end{eqnarray}
Also, by Theorem \ref{t 2.3}, assertions (\ref{1})--(\ref{5}) and $C^{1,\alpha}$ convergence within the positive set,  we obtain 
\begin{eqnarray*}
-\int_{\left\lbrace u_{0} > 0\right\rbrace} \mathcal{A}\left( X_{0}, \nabla u_{0} \right) \cdot \nabla \zeta dX & = & -\lim_{j \to \infty} \int_{\left\lbrace u_{0} > 0\right\rbrace} \mathcal{A}\left( X_{j} + r_{j}X, \nabla u_{j} \right) \cdot \nabla \zeta dX \\ \nonumber
& = & \lim_{j \to \infty} r^{m}_{j}\int_{\lbrace u_{j} > 0 \rbrace}mf\left( X_{j} + r_{j}X \right)u^{m-1}_{j}\zeta dX\\ \nonumber
& = & 0,
\end{eqnarray*}
for all $ \zeta \in C^{\infty}_{0}\left( \left\lbrace u_{0} > 0\right\rbrace \right)$. In fact,
\begin{eqnarray*}
\Big| r^{m}_{j}\int_{\lbrace u_{j} > 0 \rbrace}mf\left( X_{j} + r_{j}X \right)u^{m}_{j}\zeta dX \Big| \leq r^{m}_{j}C, 
\end{eqnarray*}
$\forall \ \zeta \in C^{\infty}_{0}\left( \left\lbrace u_{0} > 0\right\rbrace \right)$. Hence, 
\begin{eqnarray}
\div\left( \mathcal{A}\left( X_{0}, \nabla u_{0} \right) \right)= 0 \ \mbox{in} \  \lbrace u_{0} > 0 \rbrace.
\end{eqnarray}
\end{remark}

\bigskip

We  are in position to obtain the blow-up minimization problem, i.e., the minimization feature of limiting blow-up functions with respect to $B_{r_{j}}\left( X_{j} \right) $, $r_j \searrow 0$, defined as
$$
	u_j(Y) := \dfrac{1}{r_j} u(X_j + r_j Y).
$$

\begin{lemma}
\label{le 5.3}
If $u\left( X_{j} \right)=0$, $X_{j} \rightarrow X_{0} \in \Omega$, then any blow up limit $u_{0}$ with respect to $B_{r_{j}}\left( X_{j} \right)$ is minimizer of the functional
\begin{eqnarray}
\label{f l 1.1}
\mathfrak{F}_{0}\left( v \right):= \int_{B_{1}\left( 0 \right)} G\left( X_{0}, \nabla v\right) + Q\left( X_{0}\right)  \chi_{\left\lbrace v > 0 \right\rbrace } dX. 
\end{eqnarray}
\end{lemma}
\begin{proof}
Set $D = B_{1}\left( 0\right)$. Take any $v$, $v - u_{0} \in H^{1}_{0}\left( D \right)$, $\eta \in C^{\infty}_{0}\left( D \right)$, $0 \leq \eta \leq 1$. Consider
\begin{eqnarray*}
v_{j}:= v + \left( 1 - \eta \right)\left( u_{j} - u_{0}\right) \ \mbox{and} \ Q_{j}\left( X\right):= Q\left( X_{j} + r_{j}X\right), \ \ \forall X \in D. 
\end{eqnarray*}
We will denote
\begin{eqnarray}
\label{f l 1.1}
\mathfrak{F}_{j}\left( v \right):= \int_{B_{1}\left( 0 \right)} G\left( X_{j} + r_{j}X, \nabla v\right) + r^{m}_{j}f\left( X_{j} + r_{j}X \right)\left( v^{+}\right)^{m} + Q_{j}\left( X \right) \chi_{\left\lbrace v > 0 \right\rbrace } dX. 
\end{eqnarray}
Since $v_{j} = u_{j}$ in $\partial D$ and $u$ is local minimum we have for large $j$
\begin{eqnarray}
\label{l 5.3}
\mathfrak{F}_{j}\left( u_{j} \right) &\leq & \nonumber \mathfrak{F}_{j}\left( v_{j} \right).
\end{eqnarray}
From the fact that $\vert \nabla u_{j} \vert \leq C$ and $\nabla u_{j} \rightarrow \nabla u_{0} \ \ \mbox{a.e.}$, we conclude 
\begin{eqnarray*}
\int_{D} G\left( X_{j} + r_{j}X, \nabla u_{j}\right) dX \rightarrow \int_{D}  G\left( X_{0} , \nabla u_{0}\right) dX. 
\end{eqnarray*}
Similarly
\begin{eqnarray*}
\int_{D} G\left( X_{j} + r_{j}X, \nabla v_{j}\right) dX \rightarrow \int_{D} G\left( X_{0}, \nabla v_{0}\right) dX. 
\end{eqnarray*}
Moreover, we have
\begin{eqnarray*}
\Big|\int_{D} r^{m}_{j}f\left(  X_{j} + r_{j}X \right)\left(  u^{m}_{j} - \left( v^{+}_{j}\right)^{m}\right)dX \Big| &\leq & r^{m}_{j}C \int_{D} \left( \vert u_{j} \vert^{p*} + \vert v_{j} \vert^{p*}\right) dX.   
\end{eqnarray*}
Since there exists a universal constant $C>0$ such that $\Vert u_{j} \Vert_{L^{\infty}} \leq C$ (and $\Omega$ is bounded) we have
\begin{eqnarray}
 \Vert u_{j} \Vert_{L^{p*}} \leq C.
\end{eqnarray}
By definition of $v_{j}$ 
\begin{eqnarray}
 \Vert v_{j} \Vert_{L^{p*}} \leq C \Vert v \Vert_{L^{p*}} + C  \left( \Vert u_{j} \Vert_{L^{p*}} +  \Vert u_{0} \Vert_{L^{p*}}\right).
\end{eqnarray}
Therefore,
\begin{eqnarray*}
\Big|\int_{D} r^{m}_{j}f\left(  X_{j} + r_{j}X \right)u^{m}_{j} - r_{j}f\left(  X_{j} + r_{j}X\right)\left( v^{+}_{j}\right)^{m}dX \Big| \rightarrow 0.   
\end{eqnarray*}
Also, we have
\begin{eqnarray*}
\Big|\int_{D} \left( Q_{j} - Q\left( X_{0} \right) \right) \left( \chi_{\left\lbrace u_{j} > 0 \right\rbrace} - \chi_{\left\lbrace v_{j} > 0\right\rbrace} \right) dX \Big | \leq 2\int_{D} \vert Q_{j} - Q\left( X_{0} \right)\vert dX, 
\end{eqnarray*}
and using the continuity of the function $Q$ we obtain
\begin{eqnarray*}
\int_{D} \left( Q_{j} - Q\left( X_{0} \right) \right) \left( \chi_{\left\lbrace u_{j} > 0 \right\rbrace} - \chi_{\left\lbrace v_{j} > 0\right\rbrace} \right) dX \rightarrow 0. 
\end{eqnarray*}
Finally,
\begin{eqnarray*}
\chi_{\left\lbrace v_{j} > 0 \right\rbrace} \leq \chi_{\left\lbrace v > 0\right\rbrace}  + \chi_{\left\lbrace \eta < 1 \right\rbrace}. 
\end{eqnarray*}
and (see (\ref{4})) 
\begin{eqnarray*} 
\int_{D}\chi_{\left\lbrace u_{j} > 0\right\rbrace} \rightarrow \int_{D}\chi_{\left\lbrace u_{0} > 0\right\rbrace}dX. 
\end{eqnarray*}
Then follows from (\ref{l 5.3}) that
\begin{eqnarray*}
\int_{D}  G\left( X_{0}, \nabla u_{0} \right) + Q\left( X_{0}\right)\chi_{\left\lbrace u_{0} > 0 \right\rbrace} dX \leq \int_{D} G\left( X_{0}, \nabla v \right) + \left( \chi_{\left\lbrace v > 0\right\rbrace}  + \chi_{\left\lbrace \eta < 1 \right\rbrace} \right)Q\left( X_{0}\right) dX. 
\end{eqnarray*}
Taking $\eta \to 1$ finishes up the proof.
\end{proof}

\bigskip

As a consequence, we can classify blow-ups at points in the reduced free boundary.

\begin{theorem}\label{FBC pointwise sense} Let $X_0 \in
\redbdry \{ u > 0 \}$. Then, for any $X \in
\{ u > 0 \}$ near $X_0$, we have
$$
     u (X) = \alpha(X_0) \left  \langle X - X_0, \nu(X_0)
    \right \rangle^{+}  + o(|X-X_0|),
$$
where  $\nu(X_0)$ is the theoretical normal vector to $\redbdry \{u > 0 \}$ at $X_0$ and
\begin{equation}\label{FBC - red}
	\alpha(X_0) = \sqrt[p-1]{\frac{Q (X_0)}{  \mathcal{A}\left( X_0,\nu(X_0)\right)  \cdot \nu(X_0)  }}.
\end{equation}
\end{theorem}

\begin{proof}
Enhancing the notation used in  Remark \ref{r 3} we have, from standard geometric-measures arguments together with nondegeneracy and assertions (\ref{1})--(\ref{5}),
\begin{eqnarray*}
u_{0} \equiv 0 \ \mbox{in} \ \left\lbrace X \in \mathbb{R}^{n}: \langle X, \nu \left( X_{0}\right)\rangle < 0 \right\rbrace \ \ \mbox{and} \ \ \left\lbrace u_{0} > 0 \right\rbrace = \left\lbrace X \in \mathbb{R}^{n}: \langle X, \nu \left( X_{0}\right)\rangle < 0 \right\rbrace. 
\end{eqnarray*}
Also we have
\begin{eqnarray}
\div\left( \mathcal{A}\left( X_{0}, \nabla u_{0} \right) \right)= 0 \ \mbox{in} \  \lbrace u_{0} > 0 \rbrace.
\end{eqnarray}
Since $\partial \left\lbrace u_{0} > 0 \right\rbrace$ is the smooth surface $\left\lbrace X \in \mathbb{R}^{n}: \langle X, \nu \left( X_{0}\right)\rangle = 0 \right\rbrace$ we obtain
\begin{eqnarray}
\partial \left\lbrace u_{0} > 0 \right\rbrace = \redbdry \{u_{0} > 0 \}.
\end{eqnarray}
By Lemma \ref{le 5.3} and the Theorem \ref{t 5.2} ($f=0$) we find
\begin{eqnarray}
\div\left( \mathcal{A}\left( X_{0}, \nabla u_{0} \right) \right) = Q\left( X_{0}\right) \lfloor \left\lbrace X \in \mathbb{R}^{n}: \langle X, \nu \left( X_{0}\right)\rangle = 0 \right\rbrace.
\end{eqnarray}
Hence, we reach the following conclusion
\begin{eqnarray}
\label{th 5.3.2}
\nabla u_{0}\left( X\right) \cdot \nu\left( X_{0}\right) = \alpha \left( X_{0}\right), \ \ \forall X \in \left\lbrace  \langle X, \nu \left( X_{0}\right)\rangle = 0 \right\rbrace.
\end{eqnarray}
Define the function $v_{0}$ by
\begin{eqnarray}
\label{t 5.3.1}
 v_{0}\left( X \right) = \left \{
\begin{array}{lll}
 u_{0}\left( X \right), &\mbox{if} &  X \in \left\lbrace \langle X, \nu \left( X_{0}\right)\rangle < 0 \right\rbrace  \\
 - u_{0}\left( X^{*} \right), &\mbox{if} &  X \in \left\lbrace \langle X, \nu \left( X_{0}\right)\rangle \geq 0 \right\rbrace, \\    
\end{array}
\right.
\end{eqnarray}
where $X^{*}$ is the reflation of $X$ with respect to the hiperplane $\left\lbrace \langle X, \nu \left( X_{0}\right)\rangle = 0 \right\rbrace$. \\
Using standard arguments we verify that $v_{0}$ is Lipschitz continuous in $\mathbb{R}^{n}$ ($u_{0}$ is Lipschitz continuous - Lemma \ref{le 5.3} and Theorem \ref{t 3.1} with $f=0$) and 
\begin{eqnarray}
\div\left( \mathcal{A}\left( X_{0}, \nabla v_{0} \right) \right)= 0 \ \mbox{in} \  \mathbb{R}^{n}.
\end{eqnarray}
By $C^{1, \beta}$ regularity of $v_{0}$, we can to apply the blow-up argument from \cite{KSZ} to conclude that $v_{0}$ is an affine function. Then, using (\ref{th 5.3.2}) we find
\begin{eqnarray}
 u_{0} (X) = \alpha(X_0) \left  \langle X - X_0, \nu(X_0)
    \right \rangle^{+}
\end{eqnarray}
and the result is proved.     
\end{proof}

\section{Jet flow problems and smoothness of the free boundary}

In this section we address the question of smoothness of the free boundary. Per primary motivations that come from heterogeneous jet flow theory, in this section we shall only treat non-degenerate problem, i.e.,  we will work under the following assumptions:
\begin{eqnarray}\label{jet flow}
	F\left( X, \xi \right) =  \dfrac{1}{2} A(X)|\xi|^2  + f(X) (u^{+})^m + Q \chi_{\{u > 0 \}},
\end{eqnarray}
for $X \in \Omega$ and $\xi \in \mathbb{R}^{n}$, $1\le m < 2$, $f \in C\left( \Omega \right)$ and $Q \in C^{0,\beta}$, $0 < \epsilon < Q < \epsilon^{-1}$. The matrix $A$ is assume to be Lipschitz and positive definite. 

\par

The proof we will present for smoothness of the reduced free boundary  is based on flatness improvement coming from Harnack type estimates and it follows closely the recent work of \cite{DeSilva}. There are few subtle differences though. For instance the equation we work on is naturally in divergence form, thus it presents drift terms in non-divergence form. Also the free boundary condition obtained in \eqref{FBC - red} is a bit more involved then the one treated in \cite{DeSilva}. For sake of completeness and readers' convenience, we shall carry out all the details. 

\par

We shall use Caffarelli's viscosity solution setting to access the free boundary regularity theory.  Let us recall some terminologies.  Let $u, \phi \in C\left( \Omega \right)$. If $u\left( X_{0}\right) = \phi\left( X_{0} \right)$ and  there exists a neighborhood $V$ of $X_{0}$ such that
\begin{eqnarray}
\label{d 5.1}
u\left( X \right) \geq \phi\left( X \right) \ \ (\mbox{resp.} \ u\left( X \right) \leq \phi\left( X \right)) \ \ \mbox{in} \ \ V,
\end{eqnarray} 
we say that $\phi$ touches $u$ by below (resp. above) at $X_{0} \in \Omega$. Moreover, if inequality in (\ref{d 5.1}) is strict in $V \setminus \left\lbrace X_{0} \right\rbrace$, we say that $\phi$ touches $u$ strictly by below (resp. above) at $X_{0} \in \Omega$.

Next Proposition is classical in the theory of Caffarelli's viscosity solution, see \cite{C1, C2, C3}; therefore we omit its proof.

\begin{proposition} \label{prop min satisfies visc}
Assume \eqref{jet flow}. A minimizer $u$ to (\ref{Funct Intro}) is a viscosity solution to 

\begin{eqnarray}
\label{P 5.1}
	\left \{ 
		\begin{array}{lll}
			\div  \left ( A\left( X \right) \nabla u \right )  =  mf(X)u^{m-1} , & \text{ in } & \Omega_{+}\left( u \right):= \{ u > 0 \} \\
			\langle A \nabla u, \nabla u \rangle= Q & \text{ on } & \mathcal{F}(u) := \partial \{u > 0 \} \cap \Omega. 
		\end{array}
	\right.
\end{eqnarray}
The free boundary condition above is understood in the Caffarelli's viscosity sense: if $\phi \in C^{2}\left( \Omega \right)$ and $\phi^{+}$ touches $u$ by below (resp. above) at $X_{0} \in \mathcal{F}\left( u \right)$ with $\vert\nabla \phi \vert\left( X_{0} \right) \neq 0$ then
\begin{eqnarray}
\label{d 5.2.2}
\langle A\left( X_{0} \right)\nabla \phi \left( X_{0} \right), \nabla \phi \left( X_{0} \right)\rangle \leq Q (X_0) \ \ \ (\mbox{resp.} \geq Q(X_0)). 
\end{eqnarray} 
\end{proposition}

\medskip

The free boundary regularity  result  we will prove is this Section is the following:

\begin{theorem} 
\label{t 5.1}
Let $u$ be a viscosity solution to (\ref{P 5.1}) in ball $B_{1}\left( 0\right)$. Suppose that $0 \in \mathcal{F}\left( u \right) $, $Q\left( 0 \right)= 1$ and $a_{ij}\left( 0 \right) = \delta_{ij}$. There exists a universal constant $\tilde{\varepsilon} > 0$ such that, if the graph of $u$ is $\tilde{\varepsilon}$-flat in $B_{1}\left( 0\right)$, i.e.  
\begin{eqnarray}
\label{t 5.1.1}
\left( X_{n} - \tilde{\varepsilon}\right)^{+} \leq u \left( X \right) \leq \left( X_{n} + \tilde{\varepsilon} \right)^{+} \ \ \mbox{for} \ \ X \in B_{1}\left( 0 \right),
\end{eqnarray} 
and
\begin{eqnarray}
\label{t 5.1.2}
\left[ a_{ij}\right]_{C^{0,1}\left( B_{1}\left( 0\right) \right) } \leq \tilde{\varepsilon}, \ \ \Vert f \Vert_{L^{\infty}\left( B_{1}\left( 0\right) \right) } \leq \tilde{\varepsilon}, \ \ \left[ Q \right]_{C^{0,\beta}\left( B_{1}\left( 0\right) \right)} \leq \tilde{\varepsilon},
\end{eqnarray}
then $F\left( u \right)$ is $C^{1,\gamma}$ in $B_{\frac{1}{2}}\left( 0 \right)$.
\end{theorem}

\begin{corollary} Assume \eqref{jet flow}. The reduced free boundary of a minimizer $u$ to (\ref{Funct Intro}) is locally a $C^{1,\gamma}$ surface. In particular, 
$$
	\langle A(Z) \nabla u(Z), \nabla u(Z) \rangle = Q(Z),
$$
in the classical sense for $\mathcal{H}^{n-1}$ almost all free boundary points $Z \in \partial \{u > 0 \}$.
\end{corollary}

The approach will be fundamentally based on comparison criterion.

\begin{definition}
\label{d 5.3}
Let $v \in C^{2}\left( \Omega \right)$. Fixed a viscosity solution $u$ to  (\ref{P 5.1}), we say  $v$ is a strict (comparison) subsolution (resp. supersolution) to (\ref{P 5.1}) in $\Omega$, if  the following teo conditions are satisfied:
\begin{enumerate}
\item $ \div  \left ( A\left( X \right) \nabla v \right )  >  mf(X)(u^{+})^{m-1}$ (resp. $<$) in $\Omega_{+}\left( v \right)$; 
\item If $X_{0} \in \mathcal{F}\left( u \right)$ then
\begin{eqnarray*}
\label{}
\langle A\left( X_{0} \right)\nabla \phi \left( X_{0} \right), \nabla \phi \left( X_{0} \right)\rangle > Q (X_0) \ \ \ \left( \mbox{resp.} \ 0 < \langle A\left( X_{0} \right)\nabla \phi \left( X_{0} \right), \nabla \phi \left( X_{0} \right)\rangle <  Q (X_0)\right). 
\end{eqnarray*} 
\end{enumerate}
\end{definition}

Next lemma provides a basic comparison principle for solutions to the free boundary problem  (\ref{P 5.1}).
\begin{lemma} 
\label{l 5.1}
Let $u$ a viscosity solution to (\ref{P 5.1}) in $\Omega$. If $v$ is a strict subsolution to (\ref{P 5.1}) in $\Omega$ such that $u \geq v^{+}$ in $\Omega$. Then in  $\Omega^{+}\left( v \right) \cup \mathcal{F}\left( v \right)$ the strict inequality, $u > v^{+}$, holds.
\end{lemma}

Lemma \ref{l 5.1} yields the crucial  tool in the proof of  Theorem \ref{t 5.1}. More precisely, based on comparison principle granted in Lemma \ref{l 5.1}, we prove  a Harnack inequality estimate for solution $u$.  For $0 < \varepsilon < 1$, to be chosen later,  we can assume, by normalization and dilating variables,  the following conditions:
\begin{eqnarray}
\label{t 5.2.a} \Vert a_{ij} - \delta_{ij} \Vert_{L^{\infty}\left(\Omega \right)} \leq \varepsilon^{2}, \\ 
\label{t 5.2.b} \Vert mf(X)u^{m-1} \Vert_{L^{\infty}\left( \Omega \right)} \leq C_{1}\varepsilon^{2}, \\ 
\label{t 5.2.c} \Vert D a_{ij} \Vert_{L^{\infty}\left( \Omega \right)} \leq C_{0} \varepsilon^{2}, \\ 
\label{t 5.2.d} \Vert Q - 1 \Vert_{L^{\infty}\left(\Omega \right)} \leq \varepsilon^{2},
\end{eqnarray}
for $C_0$ and $C_1$ universal, depending only on Lipschitz norm of $a_{ij}$ and bounds for $f$ in (g2). We need of following Lemma.
\begin{lemma}
\label{l 5.2}
Let $u$ a viscosity solution to (\ref{P 5.1}) in $\Omega$, under assumptions  (\ref{t 5.2.a})--(\ref{t 5.2.d}). There exists a universal constant $\tilde{\varepsilon} > 0$ such that if  $0 < \varepsilon \leq \tilde{\varepsilon}$ and $u$ satisfies
\begin{eqnarray}
\label{l 5.2.1}
p^{+}\left( X \right) \leq u \left( X \right) \leq \left( p\left( X \right) + \sigma \right)^{+}, \ \ \vert \sigma \vert < \frac{1}{20} \ \ \mbox{in} \ \ B_{1}\left( 0 \right), \ \ p\left( X \right) = X_{n} + \sigma, 
\end{eqnarray} 
then if at $X_{0}= \frac{1}{10}e_{n}$ 
\begin{eqnarray}
\label{l 5.2.2}
u \left( X_{0} \right) \geq \left( p\left( X_{0} \right) + \frac{\varepsilon}{2} \right)^{+}, 
\end{eqnarray} 
then
\begin{eqnarray}
\label{l 5.2.3}
u \geq \left( p + c\varepsilon \right)^{+} \ \ \mbox{in} \ \ \overline{B}_{\frac{1}{2}}\left( 0 \right), 
\end{eqnarray}
for some $0 < c < 1$. Analogously, if
\begin{eqnarray}
\label{l 5.2.4}
u \left( X_{0} \right) \leq \left( p\left( X_{0} \right) + \frac{\varepsilon}{2} \right)^{+}, 
\end{eqnarray} 
then
\begin{eqnarray}
\label{l 5.2.5}
u \leq \left( p + \left( 1 - c \right)\varepsilon \right)^{+} \ \ \mbox{in} \ \ \overline{B}_{\frac{1}{2}}\left( 0 \right). 
\end{eqnarray}
\end{lemma}

\begin{proof}
The proof goes as in  \cite{DeSilva}. We will only verify the first statement, as the proof of the second one is analogous. Let $w \colon \overline{D} \rightarrow \mathbb{R}$ be defined by
\begin{eqnarray}
\label{l 5.2.6}
	w\left( X \right) = c\left( \vert X - X_{0} \vert^{-\gamma} - \left( \frac{4}{5}\right)^{-\gamma} \right),
\end{eqnarray}
where $D:= B_{\frac{4}{5}}\left( X_{0} \right)\setminus \overline{B}_{\frac{1}{40}}\left( X_{0} \right)$.
We choose $c > 0$ such that
\begin{eqnarray}
\label{l 5.2.7}
w =	\left \{ 
\begin{array}{lll}
0, & \text{on} & \partial B_{\frac{4}{5}}\left( X_{0} \right), \\
1, & \text{on} & \partial B_{\frac{1}{40}}\left( X_{0} \right). 
\end{array}
\right.
\end{eqnarray}
We compute directly,
\begin{eqnarray}
\label{l 5.2.8}
\partial_{i}w = -\gamma \left( X_{i} - X^{i}_{0} \right)\vert X - X_{0} \vert^{-\gamma - 2} 
\end{eqnarray}
and
\begin{eqnarray}
\label{l 5.2.9}
\partial_{ij}w = \gamma \vert X - X_{0} \vert^{-\gamma - 2} \left\lbrace \left( \gamma + 2 \right) \left( X_{i} -  X^{i}_{0}\right)\left( X_{j} - X^{j}_{0} \right)\vert X - X_{0} \vert^{- 2} - \delta_{ij} \right\rbrace.
\end{eqnarray}
If we label, $b_i = \partial_k a_{ij}$, from  $\Vert a_{ij} - \delta_{ij} \Vert_{L^{\infty}\left(\Omega \right)} \leq \varepsilon^{2}$ and $\Vert b_{i} \Vert_{L^{\infty}\left(\Omega \right)}\leq C_{0}\varepsilon^{2}$, we obtain, in $D$,by choosing $\gamma >0$ (universally)  large,
\begin{eqnarray}
\label{l 5.2.10}
\div  \left ( A\left( X \right) \nabla v \right )  & = & \nonumber \gamma \vert X - X_{0} \vert^{-\gamma - 2} \{ \left( \gamma + 2\right) \vert X - X_{0} \vert^{- 2}\sum_{i,j=1}^{n}a_{ij}\left( X \right)\left( X_{i} -  X^{i}_{0}\right) \cdot \\ \nonumber
&  & \left( X_{j} - X^{j}_{0} \right)  -  \sum_{i,j=1}^{n}a_{ij}\left( X \right)\delta_{ij} - \sum_{i=1}^{n}b_{i}\left( X \right) \left( X_{i} -  X^{i}_{0}\right)\} \\ \nonumber
& \geq & \gamma \vert X - X_{0} \vert^{-\gamma - 2} \left\lbrace \left( \gamma + 2 \right) - C\left( n \right)  \right\rbrace 
\\ 
& \geq & \delta_{0}, 
\end{eqnarray}
where $\delta_{0}> 0$ is a universal constant. From (\ref{l 5.2.1}) we have $u \geq p$ in $B_{1}\left( 0 \right)$. Thus,
\begin{eqnarray}
\label{l 5.2.11}
B_{\frac{1}{20}}\left( X_{0} \right) \subset B^{+}_{1}\left( u \right). 
\end{eqnarray}
Moreover, 
\begin{eqnarray}
\label{l 5.2.12}
\div  \left ( A\left( X \right) \nabla( u - p) \right )  =\div  \left ( A\left( X \right) \nabla u \right )  - b_{n} = mf(X)u^{m-1} - b_{n}, \ \ \mbox{in} \ \ B_{1/20}\left( X_{0}\right),
\end{eqnarray}
with
\begin{eqnarray}
	\Vert mf(X)u^{m-1} - b_{n} \Vert_{L^{\infty}\left( B_{1/20}\left( X_{0} \right)\right)} \leq C\varepsilon^{2}.
\end{eqnarray}
Hence, by Harnack inequality, we obtain
\begin{eqnarray}
\label{l 5.2.13}
u\left( X \right) - p\left( X \right) & \geq & \nonumber u\left( X_{0} \right) - p\left( X_{0} \right) - C\Vert mf(X)u^{m-1} - b_{n} \Vert_{L^{\infty}\left( B_{1/20}\left( X_{0} \right)\right)} \\ \nonumber
&\geq & u\left( X_{0} \right) - p\left( X_{0} \right) - C\varepsilon^{2},
\end{eqnarray} 
for all $X \in B_{\frac{1}{40}}\left( X_{0} \right)$.\\ Using (\ref{l 5.2.2}) for $\varepsilon$ sufficiently small
\begin{eqnarray}
\label{l 5.2.14}
u\left( X \right) - p\left( X \right) & \geq & c\varepsilon - C\varepsilon^{2} \geq c_{0}\varepsilon, \ \ \mbox{in} \ \ B_{1/40}\left( X_{0} \right).
\end{eqnarray}
Define
\begin{eqnarray}
\label{l 5.2.15}
v\left( X \right) = p\left( X \right) + c_{0}\varepsilon \left( w\left( X \right) - 1 \right), \ \ X \in \overline{B}_{\frac{4}{5}}\left( X_{0} \right),
\end{eqnarray}
and for $t \geq 0$,
\begin{eqnarray}
\label{l 5.2.16}
v_{t}\left( X \right) = v\left( X \right) + t, \ \ X \in \overline{B}_{\frac{4}{5}}\left( X_{0} \right).
\end{eqnarray}
By maximum principle (see (\ref{l 5.2.7}) and (\ref{l 5.2.10})) we have $w \leq 1$ in $D$. Then, extending $w$ to $1$ in $B_{\frac{1}{40}}\left( X_{0} \right)$ we find
\begin{eqnarray}
\label{l 5.2.17}
v_{0}\left( X \right) = v\left( X \right) \leq p\left( X \right) \leq u \left( X \right), \ \ X \in \overline{B}_{\frac{4}{5}}\left( X_{0} \right).
\end{eqnarray}
Consider 
$$
	t_{0}= \sup \left\lbrace t \geq 0 : v_{t} \leq u \ \ \mbox{in} \ \ \overline{B}_{\frac{4}{5}}\left( X_{0} \right) \right\rbrace.
$$ 
Assume, for the moment, that we have already verified $t_{0} \geq c_{0}\varepsilon$. From definition of $v$ we have
\begin{eqnarray}
\label{}
u\left( X \right) \geq v \left( X \right) + t_{0} \geq p\left( X \right) + c_{0}\varepsilon w\left( X \right), \ \ \forall X \in B_{\frac{4}{5}}\left( X_{0} \right).    
\end{eqnarray}
Notice that $ B_{\frac{1}{2}}\left( 0\right) \subset  B_{\frac{3}{20}}\left( X_{0} \right)$ and
\begin{eqnarray}
\label{}
w\left( X \right)  \geq	\left \{ 
\begin{array}{lll}
\left( \frac{3}{20}\right)^{-\gamma} - \left( \frac{4}{5}\right)^{-\gamma} , & \text{in} & B_{\frac{3}{20}}\left( X_{0} \right) \setminus B_{\frac{1}{40}}\left( X_{0} \right), \\
1, & \text{on} &  B_{\frac{1}{40}}\left( X_{0} \right). 
\end{array}
\right.  
\end{eqnarray}
Hence, we conclude ($\varepsilon$ small) that
\begin{eqnarray}
\label{}
u\left( X \right) - p\left( X \right) & \geq & \nonumber c\varepsilon, \ \ \mbox{in} \ \ B_{1/2}\left( 0 \right),
\end{eqnarray}
and the result is proved.

\par

Let us now prove that indeed $t_{0} \geq c_{0}\varepsilon$. For that, we suppose for the sake of contradiction that $t_{0} < c_{0}\varepsilon$. Then there would exist $Y_{0} \in \overline{B}_{\frac{4}{5}}\left( X_{0} \right)$ such that
\begin{eqnarray}
\label{}
v_{t}\left( Y_{0} \right) = u\left( Y_{0} \right).
\end{eqnarray}  
In the sequel, we show that $Y_{0} \in B_{\frac{1}{40}}\left( X_{0}\right)$. From definition of $v_{t}$ and by the fact that $w$ has zero boundary data on  $\partial B_{4/5}\left( X_{0} \right)$ we have
\begin{eqnarray}
\label{}
 	v_{t}= p - c_{0}\varepsilon + t_{0} < u \ \ \mbox{in} \ \ \partial B_{4/5}\left( X_{0} \right),
\end{eqnarray}
where we have used that $u \geq p$ and $t_{0} < c_{0}\varepsilon$. Moreover,
\begin{eqnarray}
\label{}
\div  \left ( A\left( X \right) \nabla v_t \right )   \geq \left( c_{0} \delta_{0} - \varepsilon \right) \varepsilon > \varepsilon^{2} \ \ \mbox{in} \ \ D
\end{eqnarray} 
and
\begin{eqnarray}
\label{t 5.2.18}
\vert \nabla v_{t_{0}} \vert \geq \vert \partial_{n}v \vert = \vert 1 + c_{0}\varepsilon \partial_{n}w \vert, \ \  \mbox{in} \ \  D.
\end{eqnarray} 
By radial symmetry of $w$, we have
\begin{eqnarray}
\label{t 5.2.19}
\partial_{n}w\left( X \right)  = \vert \nabla w\left( X \right) \vert \langle \nu_{X}, e_{n}\rangle , \ \ X \in D,
\end{eqnarray}  
where $\nu_{X}$ is the unit vector in the direction of $X - X_{0}$. From (\ref{l 5.2.8}) we have
\begin{eqnarray}
\label{}
\vert \nabla w \vert^{2} & = & \gamma^{2} \vert X - X_{0} \vert^{-2\left( \gamma + 2 \right) } \vert X - X_{0} \vert^{2} \\ \nonumber
& = & \gamma^{2} \vert X - X_{0} \vert^{-2\left( \gamma + 1 \right) } \\ 
& \geq & c > 0, \ \ \ \ \mbox{in} \ \ D.
\end{eqnarray}
Also we have $\langle \nu_{X}, e_{n}\rangle \geq c$ in $\left\lbrace v_{t_{0}} \leq 0 \right\rbrace \cap D$ (for $\varepsilon$ small enough). In fact, if $\varepsilon$ is small enough 
\begin{eqnarray}
\label{}
\left\lbrace v_{t_{0}} \leq 0 \right\rbrace \cap D \subset \left\lbrace p \leq c_{0}\varepsilon \right\rbrace = \left\lbrace X_{n} \leq c_{0}\varepsilon - \sigma \right\rbrace \subset \left\lbrace X_{n} < 1/20 \right\rbrace.
\end{eqnarray} 
We therefore conclude that 
\begin{eqnarray}
\label{}
\langle \nu_{X}, e_{n}\rangle & = & \nonumber \frac{1}{\vert X_{0} - X \vert}\langle X - X_{0}, e_{n} \rangle \\ \nonumber
& \geq & \frac{5}{4}\langle X - X_{0}, e_{n} \rangle  \\ \nonumber
& = & \frac{5}{4}\left( -X_{n} + \frac{1}{20} - \frac{1}{20} + \frac{1}{10}\right) \\ \nonumber
& > & \frac{1}{16}, \ \ \ \ \mbox{in} \ \ \left\lbrace v_{t_{0}} \leq 0 \right\rbrace \cap D.
\end{eqnarray}
Moreover, from $\Vert a_{ij} - \delta_{ij} \Vert \leq \varepsilon^{2}$ we have
\begin{eqnarray}
\label{t 5.2.19i}
\langle A\left( X \right)\xi, \xi \rangle \geq \vert \xi \vert^{2} \left( 1 - \varepsilon^{2}\right), \ \  \forall X \in \Omega, \forall \xi \in \mathbb{R}^{n}.
\end{eqnarray}
Therefore, from (\ref{t 5.2.18}), (\ref{t 5.2.19}) and (\ref{t 5.2.19i}) we obtain
\begin{eqnarray}
\langle A \nabla v_{t_{0}},\nabla v_{t_{0}} \rangle & \geq & \nonumber \vert\nabla v_{t_{0}} \vert^{2} - C\varepsilon^{2} \geq 1 + c_{1}\varepsilon + \varepsilon \left( c_{1} - C\varepsilon \right) + c_{1}^{2}\varepsilon^{2} > 1 + \varepsilon^{2} > Q,
\end{eqnarray}
in $\left\lbrace v_{t_{0}} \leq 0 \right\rbrace \cap D$.
In particular, we have
\begin{eqnarray}
\langle A \nabla v_{t_{0}},\nabla v_{t_{0}} \rangle > Q \ \ \mbox{in} \ \ D \cap \mathcal{F}\left( v_{t_{0}} \right).
\end{eqnarray}
Thus, $v_{t_{0}}$ is a strict subsolution in $D$ and by Lemma \ref{l 5.1} ($u$ is a viscosity solution of problem (\ref{P 5.1}) in $B_{1}\left( 0 \right)$) we conclude that $Y_{0} \in B_{\frac{1}{40}}\left( X_{0} \right)$. This is a contradiction. In fact, we would get
\begin{eqnarray}
u\left( Y_{0} \right) = v_{t_{0}}\left( Y_{0} \right) = v\left( Y_{0} \right) + t_{0} \leq p\left( Y_{0} \right) + t_{0} < p\left( Y_{0} \right) + c_{0}\varepsilon.
\end{eqnarray}
which drives us to a contradiction on  (\ref{l 5.2.14}).  Lemma \label{l 5.2} is concluded.
\end{proof}

\bigskip 

We can now establish the main tool in the proof of Theorem \ref{t 5.1}.

\begin{theorem}
\label{t 5.2.1} 
Let $u$ be a viscosity solution to (\ref{P 5.1}) in $\Omega$ under assumptions (\ref{t 5.2.a})--(\ref{t 5.2.d}). There exists a universal constant $\tilde{\varepsilon} > 0$ such that, if satisfies at some $X_{0} \in \Omega^{+}\left( u \right) \cup F\left( u \right)$, 
\begin{eqnarray}
\label{t 5.2.2}
\left( X_{n} + a _{0}\right)^{+} \leq u \left( X \right) \leq \left( X_{n} + d_{0} \right)^{+} \ \ \mbox{in} \ \ B_{r}\left( X_{0} \right) \subset \Omega,
\end{eqnarray} 
with
\begin{eqnarray}
\label{t 5.2.3}
d_{0} - a_{0} \leq \varepsilon r, \ \ \ \varepsilon \leq \tilde{\varepsilon} 
\end{eqnarray}
then
\begin{eqnarray}
\label{t 5.2.4}
\left( X_{n} + a _{1}\right)^{+} \leq u \left( X \right) \leq \left( X_{n} + d_{1} \right)^{+} \ \ \mbox{in} \ \ B_{\frac{r}{40}}\left( X_{0} \right)
\end{eqnarray}
with
\begin{eqnarray}
\label{t 5.2.5}
a_{0} \leq a_{1} \leq d_{1} \leq d_{0}, \ \ \ d_{1} - a_{1} \leq \left( 1 - c \right)\varepsilon r, 
\end{eqnarray} 
and $0 < c < 1$ universal.
\end{theorem}

\begin{proof}
With no loss of generality, we assume  $X_{0}=0$ and $r=1$. We analyze  two distinct cases:
\vspace{0,5cm}\\
1. $\vert a_{0} \vert < \frac{1}{20}$:
\vspace{0,5cm}\\
We put $p\left( X \right) = X_{n} + a_{0}$ and by (\ref{t 5.2.2}) 
\begin{eqnarray}
\label{t 5.2.6}
p^{+}\left( X \right) \leq u \left( X \right) \leq \left( p\left( X \right) + a_{0} \right)^{+} \ \  \left( d_{0} \leq a_{0} + \varepsilon \right). 
\end{eqnarray}
Thus, we can apply Lemma \ref{l 5.2} to obtain the result. 
\vspace{0,5 cm}\\
2. $\vert a_{0} \vert \geq \frac{1}{20}$:
\vspace{0,5cm}\\
Assume that $a_{0} < - \frac{1}{20}$. If we take $\varepsilon < \frac{1}{20}$, it follows that $\left( p + \varepsilon \right)^{+} = 0$. Thus $0$ belongs to the interior of zero phase of $u$, which turns into a contradiction.
\par

If $a_{0} > \frac{1}{20}$, then $B_{\frac{1}{20}}\left(0 \right) \subset B^{+}_{1}\left( u \right)$ and the result follows from Harnack inequality. In fact, if $p = x_{n} + d_{0}$
\begin{eqnarray}
\label{}
\div  \left ( A\left( X \right) \nabla  \left( p - u \right)  \right ) = b_{n} - mf(X)u^{m-1} , 
\end{eqnarray}
with $\Vert b_{n} - mf(X)u^{m-1}  \Vert_{L^{\infty}\left( B_{1/20}\left( 0 \right)\right)} \leq C\varepsilon^{2}$. Hence, since $\left( p - u \right)\left( 0 \right) = a_{0}$ we have
\begin{eqnarray}
\label{}
 p\left( X \right) - u\left( X \right) \geq c_{0}d_{0} - C \varepsilon^{2}, \ \ \left( 0 < c_{1} < 1 \right) 
\end{eqnarray} 
which implies
\begin{eqnarray}
\label{}
 u\left( X \right) \leq p\left( X \right) - d_{0}c_{0}  +  C \varepsilon^{2} = X_{n} + d_{0}\left( 1 - c_{0}\right) + C\varepsilon^{2}, \ \ \forall X \in B_{\frac{1}{40}}\left( X_{0}\right). 
\end{eqnarray} 
Let $c_{1} = d_{0} - a_{0}$. Then, 
\begin{eqnarray}
\label{}
u\left( X \right) & \leq & \nonumber  X_{n} + \left( c_{1} + a_{0} \right) \left( 1 - c_{0}\right) + C\varepsilon^{2} \\ \nonumber
& = & X_{n} + a_{0} + c_{1}\left( 1 - c_{0}\right) - a_{0}c_{0} + C\varepsilon^{2} \\ \nonumber
& \leq & X_{n} + a_{0} + c_{1}\left( 1 - c_{0}\right) - \frac{c_{0}}{20} + C\varepsilon^{2} \\ \nonumber
& \leq & X_{n} + a_{0} + c_{1}\left( 1 - c_{0}\right), 
\end{eqnarray} 
if $\varepsilon$ is small enough. Hence, if we put $a_{1}:= a_{0}$ and $d_{1} = a_{0} + c_{1}\left( 1 - c_{0}\right)$ we obtain the result.  
\end{proof}

\bigskip

From Harnack  inequality, Theorem \ref{t 5.2.1}, precisely as in \cite{DeSilva}, we obtain the following key  estimate for flatness improvement.
\begin{corollary}
\label{c 5.1}
Let $u$ be a viscosity solution to (\ref{P 5.1}) in $\Omega$ under assumptions (\ref{t 5.2.a})--(\ref{t 5.2.d}). If $u$ satisfies (\ref{t 5.2.2}) then in $B_{1}\left( X_{0} \right)$ the function $\tilde{u}_{\varepsilon}:= \frac{u - X_{n}}{\varepsilon}$ has a H\"older modulus of continuity at $X_{0}$ outside of ball of radius $\varepsilon / \tilde{\varepsilon}$, i.e. for all $X \in \left( \Omega^{+}\left( u \right) \cup F\left( u \right) \right)\cap B_{1}\left( X_{0} \right)$ with $\vert X - X_{0} \vert \geq \varepsilon / \tilde{\varepsilon}$ 
\begin{eqnarray}
\label{}
\vert \tilde{u}_{\varepsilon}\left( X \right) - \tilde{u}_{\varepsilon}\left( X_{0} \right)\vert \leq C \vert X - X_{0} \vert^{\gamma}. 
\end{eqnarray}    
\end{corollary}

\bigskip

We are ready to establish improvement of flatness. 

\begin{theorem}[Flatness Improvement] 
\label{t 5.3}
Let $u$ be a viscosity solution to (\ref{P 5.1}) in $\Omega$ under assumptions (\ref{t 5.2.a})--(\ref{t 5.2.d}). Assume that $u$ satisfies 
\begin{eqnarray}
\label{t 5.3.1}
\left( X_{n} - \varepsilon \right)^{+} \leq u \leq \left( X_{n} + \varepsilon \right)^{+} \ \ \mbox{for} \ \ X \in B_{1}\left( 0 \right), 
\end{eqnarray}
with $0 \in \mathcal{F}\left( u \right)$.  If $0 < r \leq r_{0}$ for $r_{0}$ a universal constant and $0 < \varepsilon \leq \varepsilon_{0}\left( r \right)$ then
\begin{eqnarray}
\label{t 5.3.2}
\left( X \cdot \nu - \frac{r}{2}\varepsilon \right)^{+} \leq u \leq \left( X \cdot \nu + \frac{r}{2}\varepsilon\right)^{+} \ \ \mbox{for} \ \ X \in B_{r}\left( 0 \right),
\end{eqnarray}
with $\vert \nu \vert = 1$, and $\vert \nu - e_{n}\vert \leq C \varepsilon^{2}$ for a universal constant $C$.
\end{theorem}
\begin{proof}
Again the proof goes in the lines of \cite{DeSilva}. F ix $0 < r \leq r_{0}$ with $r_{0}$ a universal constant to be chosen. Suppose by contradiction that there exists a sequence $\varepsilon_{j} \rightarrow 0$ and a sequence $u_{j}$ of solutions to $\left( \ref{P 5.1}\right)$ in $B_{1}\left( 0 \right)$ with right hand side $g_{j}$ and free boundary condition $Q_{j}$ such that
\begin{eqnarray}
\label{C 5}
\left( X_{n} - \varepsilon_{j} \right)^{+} \leq u_{j} \leq \left( X_{n} - \varepsilon_{j} \right)^{+} \ \ \mbox{for} \ \ X \in B_{1}\left( 0 \right), \ 0 \in \mathcal{F}\left( u_{j} \right), 
\end{eqnarray} 
but it does not satisfy the conclusion $\left( \ref{t 5.3.2}\right)$. Define
\begin{eqnarray*}
\tilde{u}_{j}\left( X \right)= \dfrac{u_{j}\left(  X \right) - X_{n}}{\varepsilon_{j}}, \ \ X \in \Omega_{1}\left( u_{j}\right),  
\end{eqnarray*}
where $\Omega_{\rho}:= \left( B_{1}^{+}\left( u \right) \cup F\left( u \right) \right) \cap B_{\rho}\left( 0 \right)$, for $0 < \rho < 1$. 
Then (as in \cite{DeSilva}) the graphs of the $\tilde{u}_{j}$ over $\Omega_{\frac{1}{2}}\left( u_{j}\right)$ converge (up to subsequence) in the Hausdorff distance to the graph of a H\"older continuous function $\tilde{u}$ over $B_{\frac{1}{2}}\left( 0 \right) \cap \left\lbrace X_{n} \geq 0 \right\rbrace$. We claim that $\tilde{u}$ is a solution of the problem
\begin{eqnarray}
\label{}
\left \{
\begin{array}{lll}
\Delta u = 0 \ \ \mbox{in} \ \ B_{\frac{1}{2}}\left( 0 \right) \cap \left\lbrace  X_{n} > 0\right\rbrace, \\
\partial_{n}\tilde{u} = 0 \ \ \mbox{on} \ \ B_{\frac{1}{2}}\left( 0 \right) \cap \left\lbrace X_{n} = 0\right\rbrace, \\
\end{array}
\right.
\end{eqnarray}
in sense viscosity (see \cite{DeSilva}, Def. 2.5 and its remark). Given a quadratic polynomial $P \left( X \right)$ touching $\tilde{u}$ at $X_{0} \in B_{\frac{1}{2}}\left( 0 \right)  \cap \left\lbrace  X_{n} \geq 0 \right\rbrace$ strictly by below we need to prove that
\begin{enumerate}
	\item[(i)] If $X_{0} \in \ B_{\frac{1}{2}}\left( 0 \right) \cap \left\lbrace  X_{n} > 0 \right\rbrace$ then $\Delta P \leq 0$;
	\item[(ii)] if $X_{0} \in \ B_{\frac{1}{2}}\left( 0 \right) \cap \left\lbrace  X_{n} = 0 \right\rbrace$ then $\partial_{n} P\left( X_{0} \right) \leq 0$.
\end{enumerate}

As in \cite{DeSilva}, there exist points $X_{j} \in \Omega_{\frac{1}{2}}\left( u_{j}\right)$, $X_{j} \rightarrow X_{0}$, and constants $c_{j} \rightarrow 0$ such that
\begin{eqnarray}
u_{j}\left( X_{j}\right) = \tilde{P}\left( X_{j} \right) 
\end{eqnarray}
and
\begin{eqnarray}
u_{j}\left( X\right)\geq \tilde{P}\left( X \right) \ \ \mbox{in a neighborhood of} \ X_{j}   
\end{eqnarray}
where 
\begin{eqnarray}
\tilde{P}\left( X \right) = \varepsilon_{j}\left( P\left( X \right) + c_{j}  \right) + X_{n}.   
\end{eqnarray} 
We have two possibilities:\\

\noindent (a) If $X_{0} \in B_{\frac{1}{2}} \cap \left\lbrace  X_{n} > 0 \right\rbrace$ then, since $P$ touches $u_{j}$ by below at $X_{j}$, we obtain
\begin{eqnarray*} 
C_{1}\varepsilon_{j}^{2}\geq g_{j}\left( X_{j}\right) & \geq & \sum_{i,l=1}^{n} a^{j}_{il}\left( X \right) \partial_{il}\tilde{P} + \sum_{i=1}^{n} b^{j}_{i}\left( X_{j} \right) \partial_{i}\tilde{P} \\ \nonumber
& = & \varepsilon_{j}\sum_{i,l=1}^{n} a^{j}_{il}\left( X_{j}\right) \partial_{il}P + \varepsilon_{j}\sum_{i=1}^{n} b^{j}_{i}\left( X_{j} \right) \partial_{i}P + b^{j}_{n}\left( X_{j}\right),
\end{eqnarray*}
where $\Vert b^{j}_{i} \Vert_{L_{\infty}} \leq C_{0}\varepsilon^{2}_{j}$ and $\Vert \partial_{i} P \Vert_{L_{\infty}} \leq C$. Therefore,
\begin{eqnarray*} 
\sum_{i,l=1}^{n} a^{j}_{il}\left( X_{j}\right) \partial_{il}P \leq C\varepsilon_{j},
\end{eqnarray*} 
Thus, we have
\begin{eqnarray*} 
\Delta P & = & \sum_{i,l=1}^{n}\left( \delta_{il} - a^{j}_{il}\left( X_{k} \right) \right) \partial_{il}P + \sum_{i,l=1}^{n} a^{j}_{il}\left( X_{k} \right) \partial_{il}P\\ \nonumber
& \leq & C\varepsilon_{j}. 
\end{eqnarray*} 
Hence, taking $j \rightarrow \infty$ we obtain
\begin{eqnarray*} 
\Delta P \leq 0.
\end{eqnarray*}

\noindent (b)  If $X_{0} \in \ B_{\frac{1}{2}} \cap \left\lbrace  X_{n} = 0 \right\rbrace$ we can assume,  see \cite {DeSilva}, that 
\begin{eqnarray}
\label{strict subharmonic}
\Delta P > 0 
\end{eqnarray}
Notice that for $j$ sufficiently large we have $X_{j} \in F\left( u_{j}\right)$. In fact, suppose by contradiction that there exists a subsequence $X_{j_{n}} \in B^{+}_{1}\left( u_{j_{n}}\right)$ such that $X_{j_{n}} \rightarrow X_{0}$. Then arguing as in (i) we obtain
\begin{eqnarray*} 
\Delta P \leq  C\varepsilon_{j}, 
\end{eqnarray*}
which contradicts (\ref{strict subharmonic}) as $j_{n} \rightarrow \infty$. Therefore, there exists $j_{0} \in \mathbb{N}$ such that $X_{j} \in F\left( u_{j}\right)$ for $j \geq j_{0}$. 
Moreover, 
\begin{eqnarray*} 
\vert \nabla \tilde{P} \vert \geq 1 - \varepsilon_{j}\vert \nabla P\vert > 0, 
\end{eqnarray*}
for $j$ sufficiently large (we can assume that $j \geq j_{0}$). Since that $\tilde{P}^{+}$ touches $u_{j}$ by below we have
\begin{eqnarray*} 
\langle A^{j}\left( X_{j}\right)\nabla \tilde{P}\left( X_{j}\right), \nabla \tilde{P}\left( X_{j}\right) \rangle \leq Q_{j}\left( X_{j}\right) \leq \left(1 + \varepsilon^{2}_{j} \right). 
\end{eqnarray*} 
Moreover,
\begin{eqnarray}
\langle A^{j}\left( X_{j} \right) \nabla \tilde{P}\left( X_{j}\right),\nabla \tilde{P}\left( X_{j}\right) \rangle & \geq & \nonumber \vert \nabla \tilde {P}\left( X_{j} \right) \vert^{2} - C\varepsilon^{2}_{j} \\
&=& \varepsilon^{2}_{j}\vert \nabla P \left( X_{j}\right) \vert^{2} + 1 + 2\varepsilon_{j}\partial_{n}P \left( X_{j}\right) - C\varepsilon_{j}^{2},
\end{eqnarray}
where we have used  $\vert \nabla \tilde {P} \vert^{2} \leq C$.
In conclusion, we obtain
\begin{eqnarray}
\label{t 5.3.4} 
\varepsilon^{2}_{j}\vert \nabla P \left( X_{j}\right) \vert^{2} + 1 + 2\varepsilon_{j}\partial_{n}P \left( X_{j}\right) - C\varepsilon_{j}^{2}\leq 1 + \varepsilon^{2}_{j}. 
\end{eqnarray}
Hence, dividing (\ref{t 5.3.4}) by $\varepsilon_{j}$ and taking $j \rightarrow \infty$ we obtain $\partial_{n}P\left( X_{0}\right) \leq 0$.    

The choose of $r_{0}$ and the conclusion of the Theorem follows from the regularity of $\tilde{u}$. 
\end{proof}

\bigskip

We can finally conclude the proof of Theorem   \ref{t 5.1}. \\

\begin{proof}[Proof of Theorem \ref{t 5.1}]
Let
\begin{eqnarray}
u_{j}\left( X \right) = \frac{u\left( \rho_{j}X\right) }{\rho_{j}}, \ \ X \in B_{1}\left( 0 \right),   
\end{eqnarray}
the sequence of rescaling with $\rho_{j}= \tilde{r}^{j}$ for a fixed $\tilde{r}$ such that
\begin{eqnarray}
\tilde{r}^{\beta} \leq \frac{1}{4} \ \ \mbox{and} \ \ \tilde{r} \leq r_{0}   
\end{eqnarray}
where $r_{0}$ is the universal constant as in Theorem \ref{t 5.3}. Notice that $u_{j}$ is solution of Problem \eqref{P 5.1} with 
\begin{eqnarray}
	\nonumber a^{j}_{il}\left( X \right) := a^{j}_{il}\left( \rho_{j}X \right),  \\ 
	\nonumber g_{j}\left( X \right) := \rho_{j}m f\left( \rho_{j} X \right) u^{m-1}\left( \rho_{j} X \right), \\ 
	\nonumber Q_{j}\left( X \right) := Q\left( \rho_{j}X \right).     
\end{eqnarray}
Moreover, if $\tilde{\varepsilon} := \varepsilon^{2}_{0}\left( \tilde{r}\right)$ and $\varepsilon_{j}:= 2^{-j}\varepsilon_{0}\left( \tilde{r}\right)$ we obtain
\begin{eqnarray}
\vert a^{j}_{il}\left( X \right) - \delta_{il} \vert = \vert a_{il}\left( \rho_{j}X \right) - a_{il}\left( 0\right)\vert \leq  \left[ a_{il}\right]_{C^{0,1}}\rho_{j} \leq \tilde{\varepsilon}\tilde{r}^{j} \leq \varepsilon^{2}_{j},     
\end{eqnarray}
\begin{eqnarray}
\Vert b^{j}_{i}\Vert_{L^{\infty}} \leq  C_{0}\left[ a_{ij}\right]_{C^{0,1}}\rho_{j} \leq C_{0}\tilde{\varepsilon}\tilde{r}^{j} \leq C_{0} \varepsilon^{2}_{j},     
\end{eqnarray}
\begin{eqnarray}
\Vert g_{j}\Vert_{L^{\infty}} \leq  \Vert g_{j}\Vert_{L^{\infty}}\rho_{j} \leq C_{1}\tilde{\varepsilon}\tilde{r}^{j} \leq  C_{1} \varepsilon^{2}_{j},     
\end{eqnarray}
\begin{eqnarray}
\vert Q_{j}\left( X \right) - 1 \vert = \vert Q\left( \rho_{j}X \right) - Q\left( 0\right)\vert \leq  \left[ Q\right]_{C^{0,\beta}}\rho^{\beta}_{j} \leq \tilde{\varepsilon}\tilde{r}^{j\beta} \leq \varepsilon^{2}_{j}.     
\end{eqnarray}
The proof now follows as in \cite{DeSilva}.
\end{proof}

\bibliographystyle{amsplain, amsalpha}

\medskip

\noindent \textsc{Eduardo V. Teixeira} \hfill \textsc{Raimundo Leit\~ao} \\
\noindent Universidade Federal do Cear\'a  \hfill  Universidade Federal do Cear\'a \\
\noindent Departamento de Matem\'atica \hfill Departamento de Matem\'atica\\
\noindent Campus do Pici - Bloco 914 \hfill Campus do Pici - Bloco 914\\
\noindent Fortaleza, CE - Brazil 60.455-760 \hfill Fortaleza, CE - Brazil 60.455-760\\
\noindent \texttt{eteixeira@ufc.br} \hfill  \texttt{juniormatufc@yahoo.com.br}

\end{document}